\documentclass[11 pt]{amsart}
\usepackage{amsmath, amssymb, amsthm, colonequals, mathtools}

\usepackage{parskip}
\usepackage[T1]{fontenc}
\usepackage{newtxtext, newtxmath}
\usepackage[margin=0.9in]{geometry}
\usepackage{enumitem}
\usepackage[colorlinks=true,linkcolor=red,citecolor=blue,urlcolor=black]{hyperref}

\allowdisplaybreaks

\newtheorem{theorem}{Theorem}[section]
\newtheorem{lemma}[theorem]{Lemma}
\newtheorem{corollary}[theorem]{Corollary}
\newtheorem{proposition}[theorem]{Proposition}
\newtheorem{maintheorem}{Theorem}

\theoremstyle{definition}
\newtheorem{defn}[theorem]{Definition}

\newtheorem{example}[theorem]{Example}
\newtheorem{remark}[theorem]{Remark}
\newtheorem{question}[theorem]{Question}
\newtheorem{construction}[theorem]{Construction}
\newtheorem{conjecturex}[theorem]{Conjecture}
\newenvironment{conjecture}
  {%
   \pushQED{\qed}\begin{conjecturex}}
  {\popQED\end{conjecturex}}
\numberwithin{equation}{theorem}

\def\im{\operatorname{im}}
\def\ker{\operatorname{ker}}

\def\rank{\operatorname{rank}}

\DeclareMathOperator{\width}{width}
\DeclareMathOperator{\gdi}{gdi}

\def\tor{\operatorname{Tor}}

\def\rank{\mathrm{rank}}

\def\FF{\mathbb{F}}
\def\GG{\mathbb{G}}
\def\KK{\mathbb{K}}

\def\NN{\mathbb{N}}

\def\RR{\mathbb{R}}

\def\WW{\mathbb{W}}
\def\ZZ{\mathbb{Z}}
\def\bv{\mathbf{v}}
\def\bw{\mathbf{w}}

\def\ge{\geqslant}
\def\le{\leqslant}
\def\phi{\varphi}
\def\bar{\overline}

\def\to{\longrightarrow}
\def\mapsto{\longmapsto}

\begin{document}
	
	\pagenumbering{arabic}
	
	\title[The Betti numbers and Golodness of numerical semigroup rings of Sally type]{The Betti numbers and Golodness of \\ numerical semigroup rings of Sally type} 
    \author[Manav Batavia]{Manav Batavia}
	\address{Department of Mathematics, Purdue University, West Lafayette, Indiana, USA - 47906.}
	\email{mbatavia@purdue.edu}

	\author[Omkar Javadekar]{Omkar Javadekar}
	\address{Department of Mathematics, Chennai Mathematical Institute, Chennai, India - 603103.}
	\email{omkarjavadekar@gmail.com, omkarj@cmi.ac.in}

    \author[Manohar Kumar]{Manohar Kumar}
	\address{Department of Mathematics, Indian Institute of Technology Madras, Chennai, India - 600036.}
	\email{manhar349@gmail.com}


\begin{abstract}
We determine explicit formulae for the Betti numbers of the defining toric ideals of two families of numerical semigroups of Sally type, resolving and extending beyond several conjectures of Goel--\c{S}ahin--Singh--Srinivasan. Our approach uses Ap\'ery resolutions and their specialisations. For the family
\[
\Gamma_m(n)=\langle m,m+1,\ldots,\widehat{m+n},\ldots,2m-1\rangle,
\qquad 2\leq n<m,
\]
where the hat denotes omission, we also compute the Poincar\'e series of the residue field and characterise Golodness. For fixed multiplicity $m$ and $3\leq n<m$, we show that these rings share the same rational residue field Poincar\'e series, but exhibit arbitrarily late departures from Serre's upper bound.
\end{abstract}

\maketitle

\section{Introduction}

A \emph{numerical semigroup} $S$ is a cofinite additive submonoid of the non-negative integers. Let $\{s_1,\ldots,s_e\}$ be the unique minimal generating set of $S$, where $s_1<\cdots<s_e$. One defines the \emph{multiplicity} of $S$ to be $s_1$ and the \emph{width} of $S$ to be $s_e-s_1$.  

Given a field $\KK$, the numerical semigroup ring $\KK[S]$ is the $\KK$-subalgebra of the polynomial ring $\KK[t]$ generated by the monomials $t^s$ for $s\in S$. Consider the polynomial ring $Q=\KK[x_1,\ldots,x_e]$ endowed with the grading $\deg(x_i)=s_i$. Define the ring homomorphism $\phi:Q\to \KK[S]$ induced by $x_i\mapsto t^{s_i}$. The kernel of $\phi$, denoted by $I_S$, is called the \emph{defining toric ideal} of $S$. 

Despite their explicit monomial descriptions, numerical semigroup rings can exhibit complicated homological behaviour. Their Betti numbers over the presenting polynomial ring $Q$ can be computed using the homology of simplicial complexes associated to $S$ \cite{CampilloMarijuan, BrunsHerzog}, and may depend on the characteristic of $\KK$ \cite[Theorem 2.1]{BrunsHerzog}. The number of generators of the defining ideal is unbounded even among numerical semigroups of embedding dimension four \cite{Bresinsky}. See \cite{Stamate} for a detailed survey on $\beta_i^Q(\KK[S])$, and \cite{CMS,EHOOPK,DGSSS,gsss25} for recent progress. The resolutions of the residue fields of numerical semigroup rings exhibit further complexity. For instance, the Poincar\'e series $\mathcal{P}_\KK^{\KK[S]}(z)=\sum_{i\geq0}\beta_i^{\KK[S]}(\KK)z^i$ need not be rational; explicit examples with transcendental Poincar\'e series are constructed in \cite{FrobergRoos, LLR}.

A related problem is to determine when the Betti numbers of the residue field attain Serre's upper bound. Writing $R=\KK[S]$, this happens precisely when
\[
\mathcal{P}_\KK^R(z)=\frac{(1+z)^e}{1-z(\mathcal{P}_R^Q(z)-1)},
\]
in which case $R$ is called \emph{Golod}. Numerical semigroup rings of maximal embedding dimension are Golod \cite[Corollary 3.7]{gosd24}, but this property does not hold for numerical semigroup rings in general, as evidenced by the examples with transcendental Poincar\'e series cited above. 

A numerical semigroup is of \emph{Sally type} if $\width(S)=s_1-1$, where $s_1$ is the multiplicity of $S$. The term was introduced in \cite{DGSSS}, although these semigroups can be traced back to the work of Sally on one-dimensional local rings \cite{s80} and have since been studied from several algebraic and combinatorial perspectives. In recent years, there has been renewed interest in these semigroups, including the works \cite{bcj25, DGSSS, gsss25, ss26}. We consider the families
\[
\Gamma_m(n)=\langle\{m,m+1,\ldots,2m-1\}\setminus\{m+n\}\rangle
\]
and
\[
\Gamma_m(2,n)=\langle\{m,m+1,\ldots,2m-1\}\setminus\{m+2,m+n\}\rangle,
\]
with $2\leq n<m$ and $4\leq n<m$, respectively (cf. Definition \ref{defn:Sallytype}). These are of Sally type when $n<m-1$; our results also include the endpoint $n=m-1$, where the width is smaller. In \cite{gsss25}, the authors construct minimal free resolutions of $\KK[S]$ for $S\in\{\Gamma_m(1),\:\Gamma_m(2),\:\Gamma_m(2,3),\:\Gamma_m(3,4)\}$ and make numerous conjectures describing the Betti numbers of other $\Gamma_m(n)$ and $\Gamma_m(n_1,n_2)$.

For the families considered here, we obtain explicit formulae for the Betti numbers of $\KK[S]$. For the one-gap family $\Gamma_m(n)$, we also determine the Poincar\'e series of the residue field and characterise Golodness. All our results are independent of the characteristic of the ground field $\KK$. Our first theorem resolves several of the conjectures in \cite{gsss25} (cf. Conjecture \ref{conj:GSSSconj}).

\begin{maintheorem}(cf. Theorems \ref{thm:defining-ideal-for-n}, \ref{thm:defining-ideal-for-24-25} and \ref{thm:defining-ideal-for-2n})\label{thm:mainA}
    Let $2\leq n<m.$
    \begin{enumerate}[leftmargin=*, label=(\alph*)]
        \item Let $S=\Gamma_m(n)$. Then $\beta_0(Q/I_S)=1$ and for $d\geq 1$, 
        \[
        \beta_{d}(Q/I_S) = d\binom{m}{d+1} - m\binom{m-2}{d-1} + \binom{m-n}{d-n+1}.
        \]
        \item Let $S=\Gamma_m(2,n)$ for $n=4,5$. Then $\beta_0(Q/I_S)=1$ and for $d\geq 1$, 
        \[
        \beta_d(Q/I_S)=d\binom{m-2}{d+1}.
        \]
        \item Let $S=\Gamma_m(2,n)$ for $n\geq6$. Then $\beta_0(Q/I_S)=1$ and for $d\geq1$,
        \[
        \beta_d(Q/I_S)=
        d\binom{m-2}{d+1}
        -\binom{m-3}{d-1}
        +\binom{m-n+2}{d-n+4}.
        \]
    \end{enumerate}
\end{maintheorem}

These formulae also relate to bounds in terms of the width of a numerical semigroup \cite{HerzogStamate, CMS}. Caviglia, Moscariello, and Sammartano \cite[Conjecture 1.3]{CMS} conjectured that
\begin{equation}\label{eq:bettibound}
\beta_d^Q(\KK[S])\leq d\binom{\width(S)+1}{d+1} \quad \forall d\geq1,
\end{equation}
and proved the conjecture for $e=4$ and $\width(S)\geq40$. The formulae in Theorem \ref{thm:mainA} verify the conjectured bound \eqref{eq:bettibound}.

We next turn to the minimal free resolution of the residue field. Although this resolution is infinite, we show that its ranks are governed by a rational Poincar\'e series with quadratic denominator. Moreover, for fixed $m$, this series is independent of $n$ throughout the range $3\leq n<m$.

\begin{maintheorem}(cf. Theorem \ref{thm:residue-betti-gamma-n})\label{thm:mainB}
    Let $3 \leq n <m$ and $S=\Gamma_m(n)$, and let $R=\KK[S]$. Then the Poincar\'e series $\mathcal P_{\KK}^R(z) = \sum_{i\geq 0}\beta_i^R(\KK)z^i$ of the residue field $\KK$ is given by 
    \[
    \mathcal P_{\KK}^R(z) 
    =\frac{1+z}{1-(m-2)z+z^2}.
    \]
\end{maintheorem}

Theorem \ref{thm:mainB} determines the Betti numbers of the residue field for our family in arbitrary multiplicity, agreeing in multiplicity $5$ with those predicted by the conjectured minimal free resolutions in \cite{CFJJM}.

We also show that the rings $\KK[\Gamma_m(n)]$ are rarely Golod. However, it is interesting to note that for a fixed $m$, these rings exhibit arbitrarily late first departures from Serre's bound, even though the Poincar\'e series of their residue fields have a uniform rational form. To describe this phenomenon, for a non-negatively graded $\KK$-algebra $R=Q/I$ with $Q=\KK[x_1,\ldots,x_e]$ and $I\subset(x_1,\ldots,x_e)^2$, we define the $i^{\text{th}}$ \emph{Golod defect number} $\mathcal{D}_i(R)$ to be the coefficient of $z^i$ in the power series
\[
    \frac{(1+z)^e}{1-z(\mathcal{P}_R^Q(z)-1)}-\mathcal{P}_\KK^R(z),
\]
and the \emph{Golod defect index} of $R$ to be 
\[
\gdi(R)\colonequals\inf\big(\{i\mid \mathcal{D}_i(R)>0\}\cup\{\infty\}\big).
\]

\begin{maintheorem}(cf. Theorem \ref{thm:golod})\label{thm:mainC}
    Let $m\geq4$, $2\leq n<m$, $S=\Gamma_m(n)$, and $R=\KK[S]$. Then $R$ is Golod iff $n=2$. Moreover, if $n\geq3$, then $\gdi(R)=n$ and $\mathcal{D}_n(R)=1$.
\end{maintheorem}

We now outline the strategy for proving our results. Let $m$ be the multiplicity of $S$. The \emph{Ap\'ery set} $\operatorname{Ap}(S)=\{0,a_1,\ldots,a_{m-1}\}$ consists of the smallest element of $S$ in each residue class modulo $m$. It gives rise to the \emph{Ap\'ery toric ideal} $J_S$ in $T=\KK[x_0,\ldots,x_{m-1}]$ (cf. Section \ref{sec:preliminaries}). Our approach uses the finite Ap\'ery resolutions of Braun--Gomes--Miller--O'Neill--Sobieska \cite{bgmns25} and the infinite Ap\'ery resolutions of Gomes--O'Neill--Sobieska--Torres D\'avila \cite{gosd24}. These resolutions are generally nonminimal for the families considered here. We determine the required Betti numbers by computing the ranks of their differentials after reduction modulo the homogeneous maximal ideal, using Lemma \ref{lem:alpha-gives-betti-nos}.

For the finite resolutions, we determine these ranks by analysing the positions and dependencies of the unit entries. This gives the Betti numbers of $T/J_S$. The Koszul tensor-product relationship between the resolutions of $T/J_S$ and $Q/I_S$ \cite[Proposition 4.1]{bgmns25} then yields the formulae in Theorem \ref{thm:mainA}.

For the resolutions of the residue field, we determine the ranks of the reduced differentials by describing the linear relations among their nonzero rows. We show that these relations are generated by those arising from the preceding differential, subject to suitable restrictions on the row labels. This yields the recursive rank formula in Theorem \ref{thm:residue-rank-formula}, from which we obtain the Poincar\'e series in Theorem \ref{thm:mainB} and consequently the Golodness results of Theorem \ref{thm:mainC}.

The computations also extend beyond the original Sally type families. Ap\'ery data organise numerical semigroups into faces of a Kunz cone. In Section \ref{sec:applications}, we give explicit families of numerical semigroups lying in the relative interiors of the same faces as our families, to which the corresponding Betti number formulae apply. We also construct gluings with $\NN$ and determine the resulting Betti numbers.

The article is structured as follows. In Section \ref{sec:preliminaries}, we collect necessary definitions, set up the notation we use throughout, and prove a key general lemma for extracting the Betti numbers from a given free resolution. In Section \ref{sec:definingideals}, we obtain the Betti numbers of $J_S$ and $I_S$ for $S=\Gamma_m(n)$ and $S=\Gamma_m(2,n)$. In Section \ref{sec:residue-field-gamma-n}, we compute the Betti numbers of $\KK$ over $\KK[S]$ for $S=\Gamma_m(n)$. We then present some applications of our results in Section \ref{sec:applications}, and conclude by posing further questions and conjectures in Section \ref{sec:questions}.

\textbf{Acknowledgements.}
We would like to thank Jonathan Monta\~{n}o for introducing us to resolutions of numerical semigroups via Apéry sets. We are also grateful to David Eisenbud and Victor Reiner for helpful discussions.

This work was initiated when M.K. was visiting Arizona State University (ASU). He thanks ASU and Jonathan Monta\~{n}o for their warm hospitality. 

M.B. was supported by NSF Grants DMS-2302430 and DMS-2100288, and by Simons Foundation Grant SFI-MPS-TSM-00012928. O.J. acknowledges support from a Postdoctoral Fellowship at the Chennai Mathematical Institute, as well as additional support from the Infosys Foundation. 

\textbf{AI disclosure.} We used ChatGPT 5.6 Sol as a collaborative tool to prove the exactness of the vector space complex in Lemma \ref{lem:residue-word-exactness}; specifically, the identification of the associated graded complex with an augmented simplicial chain complex originated from the AI tool. ChatGPT was also used for literature searches, and to locate typographical errors in the late stages of this draft. Apart from these uses, all mathematical arguments and the writing of the manuscript are our own. We take full responsibility for its content and correctness.

\section{Preliminaries}\label{sec:preliminaries}
\subsection{Numerical semigroups and their invariants} A numerical semigroup $S$ is a subset of the set $\ZZ_{\geq 0}$ of all non-negative integers such that $0 \in S$, $S$ is closed under addition, and $\ZZ_{\geq 0}\setminus S$ is finite. The smallest nonzero element of $S$ is called the \emph{multiplicity} of $S$. Throughout this article, we use $m$ to denote the multiplicity of $S$. 
It is easy to see that every numerical semigroup has a unique minimal finite generating set, which we denote by $\mathcal{A}(S)$. If $\mathcal{A}(S)=\{g_1,\ldots,g_s\}$ where $g_1<\cdots<g_s$, then the number $g_s-g_1$ is called the \emph{width of $S$}, and we denote it by $\width(S)$. The set 
\[
\operatorname{Ap}(S)\colonequals\{n\in S \mid n-m \notin S\}
=\{0,a_1,\ldots,a_{m-1}\}
\]
is called the \emph{Ap\'ery set} of $S$. In other words, the Ap\'ery set consists of the minimal element of $S$ in each congruence class modulo $m$, that is, each $a_i$ is the minimal element of $S$ satisfying $a_i \equiv i \pmod m$. In general, $\mathcal{A}(S)\subset\{m,a_1,\ldots,a_{m-1}\}$.

For convenience, we set $a_0=m$. Note that in particular we have
\[
S=\langle a_0,a_1,\ldots,a_{m-1}\rangle.
\] 
However, $\{a_0,\ldots,a_{m-1}\}$ need not form a minimal generating set of $S$; for instance when $a_i+a_j=a_{i+j}$ for some $0\leq i,j\leq m-1$.
We say that $S$ has \emph{maximal embedding dimension} if $\mathcal{A}(S)=\{a_0,\ldots,a_{m-1}\}$, that is, $\{a_0,\ldots,a_{m-1}\}$ is a minimal generating set of $S$.

Let $\KK$ be a field. The \emph{numerical semigroup ring} of $S$ is a subring of the polynomial ring $\KK[t]$ defined as
\[
\KK[S]\colonequals\KK[t^a\mid a\in \mathcal{A}(S)]=\KK[t^{a}\mid a\in S].
\]  
Let $T\colonequals \KK[x_0,x_1,\ldots,x_{m-1}]$ with the natural $\ZZ$-grading given by $\deg(x_i)=a_i$. The \emph{Ap\'ery toric ideal of $S$} is the kernel $J_S\colonequals \ker(\varphi)$ of the homomorphism 
\[
\varphi : \KK[x_0,x_1,\ldots,x_{m-1}] \to \KK[t], \qquad x_i \mapsto t^{a_i}.
\]
 
Let $Q\colonequals\KK[x_{i}\mid a_i\in\mathcal{A}(S)]\subset T$. The \emph{defining toric ideal of $S$} is
$I_S = J_S \cap Q$. Observe that 
\[
\KK[S]\cong\frac{T}{J_S}\cong\frac{Q}{I_S}.
\]

For every $0 \leq i,j \leq m-1$, define
\begin{equation}\label{2-1}
\begin{aligned}
c_{i,j} &\colonequals \frac{1}{m}(a_i+a_j-a_{i+j}) \geq 0,\\
\text{and}\qquad
b_{i,j} &\colonequals
\begin{cases}
c_{i,j} & \text{if } i+j\neq m,\\
c_{i,j}+1 & \text{if } i+j=m,
\end{cases}
\end{aligned}
\end{equation}
where the addition in subscripts is taken modulo $m$.
Observe that $c_{i,j}=0$ if and only if $a_i+a_j=a_{i+j}$. 
Since $m$ is the multiplicity, this equality is impossible when $i+j=m$. 
Hence, $b_{i,j}=0$ if and only if $c_{i,j}=0$. It is known that $J_S=\langle x_ix_j- x_0^{c_{ij}}x_{i+j} : 1 \leq i\leq j \leq m-1 \rangle$; see \cite[Lemma 3.2]{bgmns25}.

\subsection{A general tool for computing Betti numbers} 
In this subsection, we record a useful tool that we use to compute the Betti numbers. We note that the lemma below is a general result about free resolutions, and not specific to ideals coming from numerical semigroups.

Let $R$ be a non-negatively graded $\KK$-algebra with homogeneous maximal ideal $\mathfrak{m}$. Let $(\mathbb{E}_{\bullet},\partial_{\bullet})$ be an acyclic complex of free modules of finite rank
\[
\cdots \xrightarrow{\partial_{3}} E_{2}
\xrightarrow{\partial_{2}} E_{1}
\xrightarrow{\partial_{1}} E_{0},
\]
with $\operatorname{coker}(\partial_{1}) \cong M$. Then each boundary map $\partial_d: E_d \to E_{d-1}$ can be viewed as a matrix $M_d$ of size $\rank(E_{d-1}) \times \rank(E_d)$. Let $\overline{M_d}$ denote the corresponding matrix for the map $\partial_d \otimes \KK: E_d \otimes \KK \to E_{d-1} \otimes \KK$. In other words, the entries of $\overline{M_d}$ are the images of elements of $M_d$ modulo the maximal ideal $\mathfrak{m}$. We define $\alpha_d(\mathbb{E_{\bullet}})= \rank\left(\overline{M_d}\right)$.

\begin{lemma}\label{lem:alpha-gives-betti-nos}
    $\beta_d(M) = \rank(E_d)- \alpha_d(\mathbb{E_{\bullet}})-\alpha_{d+1}(\mathbb{E_{\bullet}})$ for all $d\geq 1$.
\end{lemma}
\begin{proof}
 Let $(\overline{\mathbb E_{\bullet}}, \overline{\partial_{\bullet}})$ denote the complex $\mathbb E_{\bullet}\otimes \KK$. 
By definition, for every $d \geq 1$, we have 
\[
\beta_d(M)=\dim_{\KK}\left( \tor_d(M, \KK)\right)= \dim_{\KK}\left(\dfrac{ \ker(\overline{\partial_d})}{\im(\overline{\partial_{d+1}})}\right) = \dim_{\KK}\left(\ker(\overline{\partial_d})\right)  -\dim_{\KK}\left( \im(\overline{\partial_{d+1}})\right). 
\]
By the rank-nullity theorem, we have $\dim_{\KK}\left(\ker(\overline{\partial_d})\right) =  \dim_{\KK}(\overline{E_d}) -\dim_{\KK}\left( \im(\overline{\partial_d})\right) = \rank(E_d)- \alpha_d(\mathbb{E}_\bullet)$. Since $\dim_{\KK}\left(\im(\overline{\partial_{d+1}})\right) = \rank(\overline{M_{d+1}}) = \alpha_{d+1}(\mathbb{E}_\bullet)$, we get
$\beta_d(M)= \rank(E_d)- \alpha_d(\mathbb{E}_{\bullet})-\alpha_{d+1}(\mathbb{E}_{\bullet})$.
\end{proof}
The above result is useful for computing the Betti numbers of a module once a free resolution of the module is known and the ranks of the matrices $\overline{M_d}$ can be determined. We will use this result throughout the article.

\subsection{Numerical semigroups of Sally type}
Let $S$ be a numerical semigroup with multiplicity $m$. We say that $S$ is of \emph{Sally type} if $\width(S)=m-1$. 
The study of these semigroups was motivated by the work of Sally (see \cite{s80}), who introduced one such numerical semigroup as an example of a Gorenstein local ring of multiplicity two more than its embedding dimension and whose associated graded ring is Cohen--Macaulay. We consider the following numerical semigroups of Sally type.

\begin{defn}\label{defn:Sallytype}
    Let $m\geq3$ be an integer.
    \begin{enumerate}[leftmargin=*, label=(\alph*)]
        \item Given an integer $1\leq n<m$, 
        \[
        \Gamma_m(n)\colonequals \bigg\langle \{m,m+1,\ldots,2m-1\}\setminus\{m+n\}\bigg\rangle.
        \]
        \item Given integers $1\leq n_1<n_2<m$,
        \[
        \Gamma_m(n_1,n_2)\colonequals \bigg\langle \{m,m+1,\ldots,2m-1\}\setminus\{m+n_1,m+n_2\}\bigg\rangle.
        \]
    \end{enumerate}
\end{defn}

\begin{remark}
    The following technicalities must be noted about Definition \ref{defn:Sallytype}.
    \begin{enumerate}[leftmargin=*, label=(\alph*)]
        \item The semigroups $\Gamma_m(m-1)$ and $\Gamma_m(n_1,m-1)$ have width less than $m-1$; hence, they are not of Sally type. We study their homological properties nonetheless.
        \item The semigroups $\Gamma_3(1,2)$ and $\Gamma_4(1,3)$, although included in the definition above, are not numerical semigroups and consequently fall outside the scope of this article.
    \end{enumerate} 
\end{remark}

\section{Betti numbers of defining toric ideals}
\label{sec:definingideals}

In this section, we compute the Betti numbers of the Ap\'ery toric ideals and defining toric ideals of $\Gamma_m(n)$ and $\Gamma_m(2,n)$. Let $S$ be a numerical semigroup of multiplicity $m$, $R=\KK[S]$, $T=\KK[x_0,\ldots,x_{m-1}]$, and $Q=\KK[x_i\mid i\in \mathcal{A}(S)]\subset T$. Let $J_S\subset T$ and $I_S\subset Q$ be the Ap\'ery toric ideal and defining toric ideal of $S$ respectively. For convenience, we use the notation $y=x_0$. We begin by recalling the Ap\'ery resolution of $J_S$ over $T$ from \cite{bgmns25}.

\begin{construction}{(\cite[Subsection 2.3]{bgmns25}).}\label{cons:ideal-res}
    Let $F_0=T$ and for each \(d = 1, \dots, m-1\), let \(F_d\) be the free \(T\)-module with basis 
    \[
    \{\,e_{i,\:A} \mid i \in [m-1], \ A \subseteq [m-1], \ |A| = d, \ i \ge \min(A)\,\},
    \]
    where \([m-1] = \{1, 2, \dots, m-1\}\), and the degree of each basis element is given by
    \[
    \deg(e_{i,\:A}) = a_i + \sum_{j \in A} a_j.
    \]
    The rank of this module is \(\operatorname{rank}(F_d) = d\binom{m}{d+1}\). Let $F_d=0$ for $d\geq m$.

    For any subset \(A \subseteq [m-1]\), we define the sign function \(\operatorname{sign}(j, A) = (-1)^t\), where \(t\) is the index of \(j\) when the elements of \(A\) are listed in increasing order:
    \[
    A = \{\ell_0 < \ell_1 < \cdots < \ell_t = j < \cdots < \ell_r\}.
    \]

    We set \(e_{0,A} = 0\). Furthermore, for \(i \in [m-1]\) such that \(i < \min(A)\), we define:
    \begin{equation}\label{eq:eiA}
    e_{i,\:A} = \sum_{j \in A} \operatorname{sign}(j, A)\, e_{j, (A \cup \{i\}) \setminus \{j\}}.
    \end{equation}
    Consequently, for every subset \(B \subseteq [m-1]\) with $|B|>1$, we obtain the relation:
    \begin{equation}\label{eq:relation}
    \sum_{i \in B} \operatorname{sign}(i, B)\, e_{i, B \setminus \{i\}} = 0.
    \end{equation}

    In this article, we encounter elements $e_{i,\:A}$ of two types: those where $i < \min(A)$ and those where $i \geq \min(A)$. We refer to the elements satisfying $i \geq \min(A)$ as \emph{valid basis elements}, and all others as \emph{invalid}.

    Interpreting the subscripts \(i+j\) modulo \(m\), we define the boundary map \(\partial_d: F_d \longrightarrow F_{d-1}\) for \(2 \le d < m\) as follows:
    \begin{equation}\label{eq:boundary}
    \partial_d(e_{i,\:A}) = \sum_{j \in A} \operatorname{sign}(j, A) \Bigl( x_j\, e_{i, A \setminus \{j\}} - y^{\,b_{i,j}}\, e_{i+j, A \setminus \{j\}} \Bigr).
    \end{equation}
    Finally, for \(d = 1\), the boundary map is defined by:
    \begin{equation}\label{eq:boundary1}
    \partial_1(e_{i,j}) = x_i x_j - y^{\,c_{i,j}} x_{i+j}.
    \end{equation}
\end{construction}

\begin{theorem}\cite[Theorem 3.4]{bgmns25} The complex $(\FF_{\bullet},\partial_{\bullet})$ constructed above provides a free $T$-resolution of $T/J_S$. \qed
\end{theorem}

The resolution constructed above is minimal if and only if the numerical semigroup is MED; see \cite[Corollary 3.5]{bgmns25}. 

Minimal free resolutions of the Ap\'ery toric ideal and the defining toric ideal of a given numerical semigroup are closely related to each other. The following result makes this connection explicit.

\begin{proposition}\cite[Proposition 4.1]{bgmns25}\label{prop:tensorwithKoszul}
A minimal free resolution of $T/J_S$ can be obtained as the tensor product of a minimal free resolution of $T/I_ST$ with a Koszul complex over $r$ elements, where $r$ is the number of nonzero elements in the Ap\'ery set that are not in the minimal generating set of $S$. \qed
\end{proposition}

In \cite{gsss25}, the authors conjecture descriptions of the Betti numbers of the defining toric ideals of certain numerical semigroups of Sally type. We restate their conjectures here for completeness.

\begin{conjecture}\label{conj:GSSSconj}
    Let $2\leq n<m.$
    \begin{enumerate}[leftmargin=*, label=(\alph*)]
        \item \cite[Conjecture 6.1]{gsss25} Let $S=\Gamma_m(n)$, $S_1=\Gamma_m(1)$, and $Q_1=\KK[x_0,x_2,\ldots,x_{m-1}]$. Then, 
        \[
        \beta_d(Q/I_S)=\begin{cases}
            \beta_d(Q_1/I_{S_1}) & \text{if } d\leq n-2\\
            \beta_d(Q_1/I_{S_1})+\binom{m-n}{d+1-n} & \text{if } n-2<d<m-2\\
            m-n & \text{if } d=m-2,
        \end{cases}
        \]
        \item \cite[Conjecture 6.3]{gsss25} Let $S=\Gamma_m(2,n)$ for $n=4,5$. Then $\beta_0(Q/I_S)=1$ and for $d\geq 1$, 
\[
\beta_d(Q/I_S)=d\binom{m-2}{d+1}.
\]
        \item \cite[Conjecture 6.2]{gsss25} Let $S=\Gamma_m(2,n)$ for $n\geq6$. Then 
\[
\beta_d(Q/I_S)=
        \begin{cases}
            1 & d=0\\
            \frac{d(m-1)}{m-d-2}\binom{m-3}{d+1} & 1\leq d\leq n-5\\
            \frac{d(m-1)}{m-d-2}\binom{m-3}{d+1}+1 & d=n-4
        \end{cases}\qedhere
\]
    \end{enumerate}
\end{conjecture}

We prove that Conjecture \ref{conj:GSSSconj} is true (see Theorems \ref{thm:defining-ideal-for-n}, \ref{thm:defining-ideal-for-24-25}, and \ref{thm:defining-ideal-for-2n}). In fact, Theorem \ref{thm:defining-ideal-for-2n} determines all the Betti numbers explicitly, extending the partial description conjectured in (c) above.

\subsection{Betti numbers of the defining toric ideals of $\Gamma_m(n)$ }\label{subsec-n}
In this subsection, we compute Betti numbers of Ap\'ery toric ideals and defining toric ideals of the Sally type numerical semigroups of the form $\Gamma_{m}(n)$ for $n\geq2$. Observe that for $\Gamma_m(n)$, 
\[
a_0=m, \quad a_i=m+i \; \text{for } i\neq0,n, \quad a_n=2m+n.
\]
Hence, $b_{ij}=0$ iff $i+j=n.$

\begin{lemma}\label{lem:alpha-d}
    Let $2\leq n<m$, $S=\Gamma_m(n)$, and $\FF_\bullet$ be the free resolution of the Ap\'ery toric quotient $T/J_S$ as in Construction \ref{cons:ideal-res}. We have 
    
\[
 \alpha_d({\FF_{\bullet}}) =\begin{cases}
        0 & {\text{if }} \ d=1 \\ \binom{m-1}{d-1} - \binom{m-n}{d-n} & {\text{if }}\  d \geq 2.\end{cases}
\]
\end{lemma}

\begin{proof}
Since all the entries of $M_1$ are in the maximal ideal, we see that $\alpha_1(\FF_\bullet)=0$. 
    
Assume that $d \geq 2$ and consider the matrix $M_d$ that represents the map $\partial_d: F_d \to F_{d-1}$. 
Note that the set   \[
    \{\,e_{i,\:A} \mid i \in [m-1], \ A \subseteq [m-1], \ |A| = d-1, \ i \ge \min(A)\,\},
    \] is a basis of $F_{d-1}$, and hence the rows of $M_d$ are indexed by this set. Recall that from Equation \eqref{eq:boundary} we have 
\begin{equation*}
    \partial_d(e_{i,\:A}) = \sum_{j \in A} \operatorname{sign}(j, A) \Bigl( x_j\, e_{i,\:A \setminus \{j\}} - y^{\,b_{i,j}}\, e_{i+j,\:A \setminus \{j\}} \Bigr).
\end{equation*}
Modulo the maximal ideal $\mathfrak{m}$, the term 
$\operatorname{sign}(j,A)\,y^{\,b_{i,j}} e_{i+j,\:A\setminus\{j\}}$ 
reduces to $\pm e_{i+j,\:A\setminus\{j\}}$ when $b_{i,j}=0$, and vanishes otherwise.
If $i+j < \min(A\setminus\{j\})$ then $e_{i+j,\:A\setminus\{j\}}$ is not a basis element
and must be expanded via the relation \eqref{eq:eiA} into a linear combination of {valid 
basis elements} $e_{\ell,\,((A\setminus\{j\})\cup\{i+j\})\setminus\{\ell\}}$.
Consequently, a unit entry appears in the row of $\overline{M_d}$ indexed by a basis 
element $e_{k,\:B}$ precisely when, after possibly applying \eqref{eq:eiA}, the 
coefficient of $e_{k,\:B}$ in the summation representing $\partial_d(e_{i,\:A})$ is $\pm 1$.

Now, fix a basis element $e_{k,\:B}$ of $F_{d-1}$. We consider two cases: $k=n$ and $k \neq n$.

\medskip
\noindent \textbf{Case 1: $k=n$.}
\par
A unit entry in row $e_{n,\:B}$ (with $|B| = d-1$ and $n \ge \min (B)$), if it exists, arises from a column $e_{i,\:A}$, where $i\in\{1,\ldots,n-1\}$ and $A\subset[m-1]$ such that $|A|=d$, $\:A=B\cup\{n-i\}$, and $i\geq\min(A).$ 

Let $j=n-i.$ Such a column $e_{i,\:A}$ is completely determined by a choice of $j\in\{1,\ldots,n-1\}$ such that $j\notin B$ and $n-j\geq\min(B\cup\{j\})$.
We distinguish two subcases.

\smallskip
\noindent \textit{Subcase 1: $\{1,\dots,n-1\} \subseteq B$.}
Clearly, there cannot exist $j\in\{1,\ldots,n-1\}$ such that $j\notin B$. Hence the row $e_{n,\:B}$ consists entirely of zero entries in $\overline{M_d}$.

\smallskip
\noindent \textit{Subcase 2: $\{1,\dots,n-1\} \not\subseteq B$.}
Define $S = \{1,\dots,n-1\} \setminus B \neq \emptyset$ and set $j = \min (S)$. Consider the column $\mathbf{c}_B = e_{n-j,\:A}$, where $A = B \cup \{j\}$.

We first verify that  $n-j \ge \min (A)$.
\begin{itemize}
  \item If $1 \notin B$ then $1 \in S$, so $j = \min (S) = 1$ and $\min(A) = 1$. The inequality $n-1 \ge 1$ holds as $n \ge 2$.%
  \item If $1 \in B$ then $1 \notin S$, so $j > 1$. Since $1 \in B$, we have $\min(B) = 1$, and thus $\min(A) = \min(B \cup \{j\}) = 1$. The inequality becomes $n-j \ge 1$, which holds because $j \le n-1$.
\end{itemize}
Thus $n-j \ge \min (A)$ and the basis element $e_{n-j,\:A}$ is well defined.

Now compute its boundary modulo $\mathfrak{m}$:
\[
\overline{\partial}_d(\mathbf{c}_B) = \sum_{\substack{k \in A \\ b_{n-j,k}=0}} -\operatorname{sign}(k,A)\, e_{(n-j)+k,\,A\setminus\{k\}}.
\]
Since $n-j \in \{1,\dots,n-1\}$, the condition $b_{n-j,\:k}=0$ holds only when $k \in \{1,\dots,n-1\}$ and $(n-j)+k = n$, i.e.\ $k = j$. Hence the sum reduces to the single term with $k=j$:
\[
\overline{\partial}_d(\mathbf{c}_B) = -\operatorname{sign}(j,\:A)\, e_{n,\:A\setminus\{j\}} = \pm e_{n,\:B}.
\]
All other terms vanish modulo $\mathfrak{m}$. Therefore the column $\mathbf{c}_B$ has exactly one nonzero entry in $\overline{M_d}$, which is a unit located in row $e_{n,\:B}$. 

We now determine the contribution of these rows to the rank of $\overline{M_d}$. Suppose we have a linear combination of {all} rows of $\overline{M_d}$ that equals the zero vector. For a specific row $e_{n,\:B}$ with $\{1,\ldots,n-1\}\not\subseteq B$, we examine the column $\mathbf{c}_B$ constructed above. Its only nonzero entry is in row $e_{n,\:B}$, and equals $\pm 1$. Thus, evaluating the linear combination of the rows of $\overline{M_d}$ at column $\mathbf{c}_B$ forces the coefficient of row $e_{n,\:B}$ to be zero. Since this holds for every such row, the rows $e_{n,\:B}$ with $\{1,\dots,n-1\}\not\subseteq B$ are linearly independent. The rows with $\{1,\dots,n-1\}\subseteq B$ (and $k=n$) are identically zero and contribute nothing to the rank. Consequently, among the rows of $\overline{M_d}$ with $k=n$, the row space has dimension equal to the number of subsets $B\subseteq[m-1]$ of size $d-1$ satisfying $n\ge \min(B)$ and $\{1,\dots,n-1\}\not\subseteq B$. Equivalently, this contribution is the cardinality of the set 
\[
\left\{ e_{k,\:B}  \mid B \subseteq [m-1],\: \vert B \vert = d-1,\:\{1, \ldots, n-1\} \not\subset B \text{ and } k=n \geq \min(B)\right\},
\]
 which equals
\[
\binom{m-1}{d-1}
-
\binom{m-n-1}{d-1}
-
\binom{m-n}{d-n}.
\]
 Moreover, the above argument also shows that each row $e_{n,\:B}$ with $\{1,\ldots,n-1\}\not\subseteq B$ is linearly independent from any subset of rows of $\bar{M_d}$ not containing $e_{n,\:B}$. Hence, each such row can be ``ignored'' for the rest of the rank computation.

\medskip
\noindent \textbf{Case 2: $k \neq n$.}
\par
\textit{Subcase 1: $k < n$.} Suppose a column of $\overline{M_d}$ contained a unit in row $e_{k,\:B}$. A unit can only arise from a term $\pm e_{\,i+j,\:A\setminus\{j\}}$ in $\overline{\partial}_d(e_{i,\:A})$ with $b_{i,j}=0$, for which we must have $i,j \in \{1,\ldots, n-1\}$ and $i+j=n$. Hence the term is $\pm e_{\,n,\:A\setminus\{j\}}$. If $n \ge \min(A\setminus\{j\})$, then $e_{n,\,A\setminus\{j\}}$ is already a valid basis vector, and the unit would appear in row $e_{n,\:A\setminus\{j\}}$, not in the row $e_{k,\:B}$. If $n < \min(A\setminus\{j\})$, then $e_{n,\,A\setminus\{j\}}$ is not a valid basis vector; we must apply the relation \eqref{eq:eiA} to rewrite it as a linear combination of basis vectors $e_{\ell,\,((A\setminus\{j\})\cup\{n\})\setminus\{\ell\}}$ with $\ell \in A\setminus\{j\}$. Since $n < \min(A\setminus\{j\}) \le \ell$, all resulting row indices are strictly greater than $n$, and none equals $k$ because $k < n$. Thus in all cases the row $e_{k,\:B}$ has no unit entry, and is identically zero in $\overline{M_d}$.

\smallskip
\textit{Subcase 2: $k > n$.}  
Suppose that the row $e_{k,\:B}$ contains a unit entry in $\overline{M_d}$.  
As before, a unit can only arise from a term $\pm e_{\,n,\,A\setminus\{j\}}$ in 
$\overline{\partial}_d(e_{i,\:A})$ with $b_{i,j}=0$, i.e.\ $i,j\in\{1, \ldots,n-1\}$ and 
$i+j=n$.  Since $k\neq n$, the basis element $e_{n,\,A\setminus\{j\}}$ cannot 
already be equal to $e_{k,\:B}$; therefore we must have  
$n < \min(A\setminus\{j\})$, so that $e_{n,\,A\setminus\{j\}}$ is expanded 
using \eqref{eq:eiA}. 

Write the expansion
\[
e_{n,\,A\setminus\{j\}} = \sum_{\ell\in A\setminus\{j\}} 
\operatorname{sign}(\ell,\,A\setminus\{j\})\,
e_{\ell,\;((A\setminus\{j\})\cup\{n\})\setminus\{\ell\}} .
\]
For the row $e_{k,\:B}$ to appear with a unit coefficient, there must exist an 
index $\ell\in A\setminus\{j\}$ such that  
$k = \ell$ and $B = ((A\setminus\{j\})\cup\{n\})\setminus\{k\}$.  
Note that this forces $k\notin B$.  

Since $n < \min(A\setminus\{j\})$, every element of $A\setminus\{j\}$ is 
strictly greater than $n$.  The set $B$ is obtained from 
$(A\setminus\{j\})\cup\{n\}$ by removing $k$; hence $n\in B$ and every other 
element of $B$ belongs to $A\setminus\{j\}$ and is therefore $> n$.  
Consequently $\min(B) = n$.  

Thus a necessary condition for a unit to appear in row $e_{k,\:B}$ with $k>n$ is  
$\min(B)=n$ and $k\notin B$.  Rows indexed by $e_{k,\:B}$ with $k>n$ and $\min(B)\neq n$ are zero in $\overline{M_d}$.

Now assume $\min(B)=n$ and $k\notin B$.  From the expansion \eqref{eq:eiA} we 
saw that the only way to obtain a unit in row $e_{k,\:B}$ is to start from the 
(not yet expanded) element $e_{n,\,A\setminus\{j\}}$ with 
$A\setminus\{j\} = (B\setminus\{n\})\cup\{k\}$.  Hence the term that, after 
applying \eqref{eq:eiA}, yields a unit in row $e_{k,\:B}$ is exactly 
$\pm e_{n,\,(B\setminus\{n\})\cup\{k\}}$.

Consequently, a column $e_{i,\:A}$ of $F_d$ can produce such a unit only if
\[
\overline{\partial}_d(e_{i,\:A}) \text{ contains } \pm e_{n,\,(B\setminus\{n\})\cup\{k\}}
\]
with $b_{i,j}=0$.  By \eqref{eq:boundary} this forces $A = (B\setminus\{n\})\cup\{k\}\cup\{j\}$ 
and $i+j=n$ for some $j\in\{1,\dots,n-1\}$ (so that 
$b_{i,j}=0$).  Moreover, $e_{i,\:A}$ must be a valid basis element of $F_d$, which requires $i \ge \min(A)$.

Since $\min(B)=n$, every element of $B\setminus\{n\}$ is strictly greater 
than $n$.  Also $k>n$, while $j < n$.  Therefore
\[
\min(A)=\min\bigl((B\setminus\{n\})\cup\{k\}\cup\{j\}\bigr)=j ,
\]
and the condition $i\ge \min(A)$ becomes $i\ge j$. In particular, as $n \geq 2$, we see that $\overline{\partial}_d(e_{n-1,\:A})$, where $A=(B\setminus \{n\})\cup\{k\}\cup\{1\}$, contains the term $\pm e_{n,\:(B\setminus\{n\})\cup\{k\}}$.

Given any $e_{k,\:B}$ with $k>n$ and $n=\min{(B)}$, let $R_{k,\:B}$ denote the set of rows of $\overline{M_d}$ containing a unit that are obtained by applying \eqref{eq:eiA} to $e_{n,(B\setminus\{n\})\cup\{k\}}$. Note that any two rows in $R_{k,\:B}$ are equal up to sign. 
Moreover, if $e_{(k_1,B_1)} \neq e_{(k_2,\:B_2)}$ with $k_1,k_2>n$, $n=\min(B_1)=\min(B_2)$, and $\{k_1\}\cup B_1 \neq \{k_2\} \cup B_2$, then $R_{k_1,B_1} \cap R_{k_2,\:B_2}= \emptyset$. 
These observations, together with the analysis done in Case 1, imply that all such distinct $R_{k,\:B}$  {contribute exactly} $1$ to the rank of $\overline{M_d}$. The number of such $R_{k,\:B}$ equals

\[
\dfrac{1}{d-1}\vert \left\{ e_{k,\:B} \mid B\subseteq [m-1], \vert B \vert = d-1, k \notin B,  {\text{  and }} k > n=\min(B)\right\}\vert = \binom{m-n-1}{d-1}.
\]

Therefore, by Cases 1 and 2 above, we get 
\[
\alpha_d=\rank(\overline{M_d})=\binom{m-1}{d-1}
-
\binom{m-n-1}{d-1}
-
\binom{m-n}{d-n} + \binom{m-n-1}{d-1}
= \binom{m-1}{d-1}
-
\binom{m-n}{d-n}.
\]
 
This completes the proof.
\end{proof}

\begin{remark}
    In the above proof, we have effectively shown that for $d\geq2$, a basis of the row space of $\bar{M_d}$ is in one-to-one correspondence with the set 
\[
\Lambda=\{ B \subseteq [m-1] \mid \vert B \vert = d-1 {\text{  and }}\{1, \ldots, n-1\} \not\subset B\},
\]
 under the correspondence 
    \[
         B \longleftrightarrow
         \begin{cases}
             e_{n,\:B} \text{ if }n\geq\min(B)\\
             e_{k,\:(B\cup\{n\})\setminus\{k\}} \text{ if } n<\min(B) =: k.
         \end{cases}
    \]
    We will use this crucial observation in the proof of Lemma \ref{lem:alphad-2n}.
\end{remark}

With the help of Lemma \ref{lem:alpha-d} we calculate the Betti numbers of the Ap\'ery toric ideal of $\Gamma_m(n)$.
\begin{theorem}\label{thm:betti-ideal}
    Let $2\leq n<m$ be integers and $S= \Gamma_m(n)$. Then for the Ap\'ery toric ideal $J_S$ of $S$, we have $\beta_0(T/J_S)=1$, and 
    \[
    \beta_d(T/J_S)= \begin{cases}
        \displaystyle\binom{m-1}{2}+\binom{m-n}{2-n} & \text{if } d=1,\\[10pt]
        \displaystyle d\binom{m}{d+1} - \binom{m}{d} + \binom{m-n+1}{d-n+1} & \text{if } d \ge 2.
    \end{cases}
    \]
\end{theorem}
\begin{proof}
Clearly, $\beta_0(T/J_S)=1$. Note that the complex $(\FF_\bullet,\partial_\bullet)$ from Construction~\ref{cons:ideal-res} is a free resolution of the Ap\'ery toric quotient $T/J_S$. Hence  by  Lemma \ref{lem:alpha-gives-betti-nos}, given any $d \geq 1$, we have
\[
\beta_d(T/J_S)=\rank(F_d)-\alpha_d(\FF_\bullet)-\alpha_{d+1}(\FF_\bullet),
\]
where $\alpha_d(\FF_\bullet)=\rank(\overline{M}_d)$ is given by Lemma~\ref{lem:alpha-d}.  
Recall that $\rank(F_d)=d\binom{m}{d+1}$ for every $d\ge 1$.

Suppose $d=1$. Then using $\alpha_1(\FF_\bullet)=0$ and the expression for $\alpha_2(\FF_\bullet)$ from Lemma~\ref{lem:alpha-d}, we get
\[
\beta_1(T/J_S)=\rank(F_1)-\alpha_2(\FF_\bullet) =\binom{m}{2}-\Bigl(\binom{m-1}{1}-\binom{m-n}{2-n}\Bigr)
 =\binom{m-1}{2}+\binom{m-n}{2-n}.
\]

Suppose $d\geq 2$. Then Lemma~\ref{lem:alpha-d} gives
\[
\alpha_d(\FF_\bullet)=\binom{m-1}{d-1}-\binom{m-n}{d-n},\qquad 
\alpha_{d+1}(\FF_\bullet)=\binom{m-1}{d}-\binom{m-n}{d+1-n}.
\]
Therefore, we get
\[
\begin{aligned}
\beta_d(T/J_S) &= d\binom{m}{d+1}
          -\Bigl(\binom{m-1}{d-1}-\binom{m-n}{d-n}\Bigr)
          -\Bigl(\binom{m-1}{d}-\binom{m-n}{d+1-n}\Bigr) \\[4pt]
        &= d\binom{m}{d+1}
           -\Bigl(\binom{m-1}{d-1}+\binom{m-1}{d}\Bigr)
           +\Bigl(\binom{m-n}{d-n}+\binom{m-n}{d+1-n}\Bigr).
\end{aligned}
\]
Using the binomial identities
\[
\binom{m-1}{d-1}+\binom{m-1}{d}=\binom{m}{d},
\qquad
\binom{m-n}{d-n}+\binom{m-n}{d+1-n}=\binom{m-n+1}{d+1-n},
\]
we obtain
\[
\beta_d(T/J_S) = d\binom{m}{d+1} - \binom{m}{d} + \binom{m-n+1}{d-n+1},
\]
which is exactly the claimed formula for $d\ge 2$.
\end{proof}

Having computed the Betti numbers of the Ap\'ery toric ideal $J_S$ of $S=\Gamma_m(n)$, we now compute the Betti numbers of the defining toric ideal $I_S$. 
By Proposition \ref{prop:tensorwithKoszul}, a minimal free resolution of \(T/J_S\) can be obtained as the tensor product of a minimal free resolution of \(T/I_ST\) with a Koszul complex on a single element. Therefore, we have the following relations among the Betti numbers of $T/J_S$ and $Q/I_S$.
\begin{equation}\label{eq:koszul_complex}
 \beta_{d}^T(T/J_S) = \beta_{d}^Q(Q/I_S)+\beta_{d-1}^Q(Q/I_S) \text{\ for all\ } d \geq 1,   \text{ and } \beta_0^T(T/J_S)=\beta_0^Q(Q/I_S).
\end{equation} 
For convenience, we drop the superscripts $T$ and $Q$. The above equation forces 
\begin{equation}\label{eq:generating-function}
\sum\limits_{i=0}^{\infty} \beta_{i}(T/J_S) z^{i} = \left(\sum_{i=0}^{\infty} \beta_{i}(Q/I_S) z^{i} \right)(1+z)
\end{equation}
Hence, from Theorem \ref{thm:betti-ideal} and the relations above, we can compute the Betti numbers of the defining ideal $I_S$. 

\begin{theorem}\label{thm:defining-ideal-for-n}
Let $2\leq n<m$ be integers and $S=\Gamma_m(n)$. Then $\beta_0(Q/I_S)=1$, and for $d \geq 1$
\begin{equation*}
\beta_{d}(Q/I_S) = d\binom{m}{d+1} - m\binom{m-2}{d-1} + \binom{m-n}{d-n+1}.
\end{equation*}
The Poincar\'e series of $R$ over $Q$ is
\[
\mathcal{P}_R^Q(z)=1+m(1+z)^{m-2}-\frac{(1+z)^m-1}{z}+z^{n-1}(1+z)^{m-n}.
\]
\end{theorem}

\begin{proof}
Clearly, $\beta_0(Q/I_S)=1$. 

Now, from Theorem \ref{thm:betti-ideal} and the fact that for $n \ge 2$, $\binom{m-n+1}{2-n} = \binom{m-n}{2-n}$, we get 
\begin{align*}
\beta_{1}(T/J_S) &= \binom{m-1}{2} + \binom{m-n}{2-n} = \binom{m}{2} - \binom{m-1}{1} + \binom{m-n+1}{2-n}  {\text{ \ \ \ and }}\\
\beta_{d}(T/J_S) &= d\binom{m}{d+1} - \binom{m}{d} + \binom{m-n+1}{d-n+1} {\text{ \ \ \ for }} d \geq 2.
\end{align*}
Using the above identities, we get the following:
\begin{equation}\label{eq:expansion}
\sum_{i \ge 0} \beta_{i}(T/J_S) z^{i} = 1 + z + \sum_{i \ge 1} i \binom{m}{i+1} z^{i} - \sum_{i \ge 1} \binom{m}{i} z^{i} + \sum_{i \ge 1} \binom{m-n+1}{i-n+1} z^{i}
\end{equation}
We have 
\begin{align*}
\sum_{i \ge 1} i \binom{m}{i+1} z^{i} &= \sum_{i \ge 0} i \binom{m}{i+1} z^{i} \\
&= \sum_{i \ge 0} \left( (i+1) \binom{m}{i+1} z^{i} - \binom{m}{i+1} z^{i} \right) \\
&= \frac{1}{z} \sum_{i \ge 0} (i+1) \binom{m}{i+1} z^{i+1} - \frac{1}{z} \sum_{i \ge 0} \binom{m}{i+1} z^{i+1} \\
&= \frac{1}{z} \left( m z (1+z)^{m-1} \right) - \frac{1}{z} \left( (1+z)^{m} - 1 \right) \\
&= m (1+z)^{m-1} - \frac{(1+z)^{m} - 1}{z}.
\end{align*}
Therefore, rewriting the three summations on the right-hand side of Equation \eqref{eq:expansion} as  
\begin{equation}
\sum_{i \ge 1} i \binom{m}{i+1} z^{i} =  m (1+z)^{m-1} - \frac{(1+z)^{m} - 1}{z},
\end{equation}
\begin{equation}
\sum_{i \geq 1} \binom{m-n+1}{i-n+1} z^{i} = z^{n-1} \sum_{i \ge 1} \binom{m-n+1}{i-n+1} z^{i-n+1} 
= z^{n-1} (1+z)^{m-n+1},
\end{equation}
and
\begin{equation}
\sum_{i \geq 1} \binom{m}{i} z^{i} = (1+z)^{m} - 1,
\end{equation}
we obtain
\begin{align*}
\sum_{i\geq 0} \beta_{i}(T/J_S) z^{i} &= 1 + z + m (1+z)^{m-1} - \frac{(1+z)^{m} - 1}{z} - \left( (1+z)^{m} - 1 \right) + z^{n-1} (1+z)^{m-n+1}.
\end{align*}
By Equation \eqref{eq:generating-function}, we get
\begin{align*}
\sum_{i \geq 0} \beta_{i}(Q/I_S) z^{i} &= \frac{1}{1+z} \left( 1 + z + m (1+z)^{m-1} - \frac{(1+z)^{m} - 1}{z} - ((1+z)^{m} - 1) + z^{n-1} (1+z)^{m-n+1} \right) \\
&= 1 + m (1+z)^{m-2} - \frac{(1+z)^{m} - 1}{z} + z^{n-1} (1+z)^{m-n}.
\end{align*}

Thus, for each $i \geq 1$, we have
\begin{align*}
\beta_i(Q/I_S)
&= m\binom{m-2}{i} - \binom{m}{i+1} + \binom{m-n}{i-n+1} \\
&= m\binom{m-1}{i} - m\binom{m-2}{i-1} - \binom{m}{i+1} + \binom{m-n}{i-n+1} \\
&= i\binom{m}{i+1} - m\binom{m-2}{i-1} + \binom{m-n}{i-n+1}.
\qedhere
\end{align*}
\end{proof}

As an application of Theorem~\ref{thm:defining-ideal-for-n}, we prove Conjecture \ref{conj:GSSSconj}(a).
\begin{corollary}\label{cor:betti-compare}
 Let $2\leq n<m$, $S=\Gamma_m(n)$, $S_1=\Gamma_m(1)$, and $Q_1=\KK[x_0,x_2,\ldots,x_{m-1}]$. Then for each $1\leq d \leq m-3$, we have
\[
\beta_d(Q/I_S)
=
\beta_d(Q_1/I_{S_1})
+
\binom{m-n}{d+1-n},
\]
and $\beta_{m-2}(Q/I_S)=m-n$.
\end{corollary}

\begin{proof} 
The statement about $\beta_{m-2}(Q/I_S)$ follows directly from Theorem~\ref{thm:defining-ideal-for-n}. Now, by 
\cite[Corollary~2.7]{gsss25}, we have the following:
\[
\beta_d(Q_1/I_{S_1}) = \frac{md}{m-d-1}\binom{m-2}{d+1}, 
\quad 1\leq d\leq m-3, 
\qquad \beta_{m-2}(Q_1/I_{S_1}) = 1.
\]
Now, for every $2 \leq n < m$ and $1 \leq d \leq m-3$, consider
\begin{align*}
\beta_d(Q_1/I_{S_1})
+
\binom{m-n}{d+1-n}
&= \frac{md}{m-d-1}\binom{m-2}{d+1}
+
\binom{m-n}{d+1-n} \\
&= \frac{md(m-2)!(m-d-2)}
{(d+1)!(m-d-1)!}
+
\binom{m-n}{d+1-n} \\
&= \frac{m(m-2)!}{(d+1)!(m-d-1)!}
\bigl(d(m-1)-d(d+1)\bigr)
+
\binom{m-n}{d+1-n} \\
&= d\frac{m!}{(d+1)!(m-d-1)!}
-
m\frac{(m-2)!}{(d-1)!(m-d-1)!}
+
\binom{m-n}{d+1-n} \\
&= d\binom{m}{d+1}
-
m\binom{m-2}{d-1}
+
\binom{m-n}{d+1-n},
\end{align*}
which agrees with the formula obtained in Theorem \ref{thm:defining-ideal-for-n}.
\end{proof}

\subsection{Betti numbers of the defining toric ideal of \texorpdfstring{$S=\Gamma_m(2,n)$}{Sm(2,n)}} \label{subsec-2n}

Let $S=\Gamma_m(2,n)$, with $4 \leq n <m$. In this case, we have $a_i=m+i$ for $i\in\{0,\ldots,m-1\}\setminus\{2,n\}$, $a_2=2m+2$, and $a_n=2m+n$. As a result, $b_{ij}=0$ iff $i=j=1$ or $i+j=n$ and $i\in\{1,\ldots,n-1\}\setminus\{2,n-2\}.$ We compute the Betti numbers of the defining toric ideals of $S$ by following the same strategy as in the previous subsection.

\begin{lemma}\label{lem:alphad-2n}
    Let $4 \leq n<m$,  $S=\Gamma_m(2,n)$, and $\FF_{\bullet}$ be the free resolution of the Ap\'ery toric quotient $T/J_S$ as in Construction \ref{cons:ideal-res}. 

    For $n \geq 6$, we have
    
\[
 \alpha_d(\FF_{\bullet}) =\begin{cases}
        0 & {\text{if }} \ d=1 \\
        2m-4 & \text{if } \ d=2 \\
        \binom{m-2}{d-1}+\binom{m-1}{d-1} - \binom{m-n+3}{d-n+3} & {\text{if }}\  d \geq 3.\end{cases}
\]
    For $n=4, 5$, we have
    
\[
 \alpha_d(\FF_{\bullet}) =\begin{cases}
        0 & {\text{if }} \ d=1 \\
        2m-5 & \text{if } \ d=2 \\
        2\binom{m-2}{d-1} & {\text{if }}\  d \geq 3.\end{cases}
\]
    
\end{lemma}
\begin{proof}
    We build on the ideas discussed in the proof of Lemma \ref{lem:alpha-d}. Let $6\leq n<m.$

    Since all the entries of $M_1$ are in the maximal ideal, we conclude that $\alpha_1(\FF_\bullet)=0$. 

    For $d\geq 2$, we fix a basis element $e_{k,\:B}$ of $F_{d-1}$ and investigate when a nonzero entry appears in the row $e_{k,\:B}$ of $\bar{M_d}$. We consider two cases: $k=n$ and $k=2$. 

    \textbf{Case 1:} $k=n$.

    A unit entry in row $e_{n,\:B}$ arises from either
    \begin{enumerate}[leftmargin=*, label=(\alph*)]
        \item a column $e_{i,\:A}$ with $i,j\in\{1,3,4,\ldots,n-3,n-1\}$, $j\in A$, $i+j=n$, and $B=A\setminus\{j\}$, or
        \item a column $e_{1,\:A}$ with $i=1$, $1,n\in A$, $2<\min(A\setminus \{1\})$, and $B=A\cup\{2\}\setminus\{1,n\}$. 
    \end{enumerate}
    Hence, if $\{1,3,4,\ldots,n-3,n-1\}\subset B$, the row $e_{n,\:B}$ is zero in $\bar{M_d}$.

    Suppose $\{1,3,4,\ldots,n-3,n-1\}\not\subset B$. Let $\Gamma=\{1,3,4,\ldots,n-3,n-1\}\setminus B$ and $i=\min(\Gamma)$. Consider the column $\mathbf{c}_B\coloneqq e_{n-i,\:A}$, where $A=B\cup\{i\}$. We need to ensure that $n-i\geq\min(A)$. Indeed, this follows as $\min(A)=1$.
    
    Then the image under $\bar\partial_d$ of the column $\mathbf{c}_B$ contains the term $\pm e_{n,\:B}$. However, we need to exercise care on whether the column $\mathbf{c}_B$ has any other nonzero entries. We split the argument into two subcases:

    \textit{Subcase 1:} $n-i\neq1.$ 

    Here, as $n-i>1$, it is not possible to have $n-i+j=2$ for $j\in A$. Therefore, the only nonzero entry appearing in the column $\mathbf{c}_B$ is in the row $e_{n,\:B}$. 

    \textit{Subcase 2:} $n-i=1.$

    Here, as $i=n-1$ and $i=\min(\Gamma)$, we must have $1\in B\subset A$. This forces the column $\mathbf{c}_B=e_{1,A}$ to have a unit in the row $e_{2,\:A\setminus\{1\}}$ in addition to the unit in the row $e_{n,\:B}$. All other entries of $\mathbf{c}_B$ are zero. 

    Assume that the row $e_{2,\:A\setminus\{1\}}$ is valid. The only unit in this row is in the column $\mathbf{c}_B$. Also, as $\{1,3,4,\ldots,n-3\}\subset B$, the only nonzero entry in the row $e_{n,\:B}$ is in the column $\mathbf{c}_B$. Hence, the rows $e_{n,\:B}$ and $e_{2,\:A\setminus\{1\}}$ are equal up to multiplication by $\pm 1$. 

    If $e_{2,A\setminus\{1\}}$ is invalid, expand it using \eqref{eq:eiA}. For each $l\in A\setminus\{1,n\}$, the resulting row $e_{l,(A\cup\{2\})\setminus\{1,l\}}$ equals the row $e_{n,B}$ up to sign. If $n\in A$, the additional row $e_{n,(A\cup\{2\})\setminus\{1,n\}}$ is counted separately in Subcase~1, since its set label does not contain $1$. 

    We will count the contribution of the row $e_{n,\:B}$ to the rank and be cautious to not count the row $e_{2,\:A\setminus\{1\}}$ later in our calculation. From the above discussion, we conclude that each row of the form $e_{n,\:B}$ with $\{1,3,4,\ldots,n-3,n-1\}\not\subset B$ contributes to the rank of $\bar{M}_d$.

    Observe that in the argument so far, we have not used the fact that $n\geq\min(B).$ If $n<\min(B)$, the row $e_{n,\:B}$ is not valid, but as in the proof of Lemma \ref{lem:alpha-d}, each of the rows $e_{l,\:B\cup\{n\}\setminus\{l\}}$ are equal up to sign and collectively contribute $1$ to the rank of $\bar{M}_d.$ Note that as $n<\min(B)$, we have $2\notin B$ and thus we are not exposing ourselves to the possibility of double-counting in the future. 

    Therefore, the contribution of the rows $e_{n,\:B}$ (valid or not) to the rank of $\bar{M}_d$ is the cardinality of the set 
\[
\{B\:|\: B\subset [m-1], \: |B|=d-1 \text{ and } \{1,3,4,\ldots,n-3,n-1\}\not\subset B\}, 
\]
 which equals 
\[
\binom{m-1}{d-1}-\binom{m-n+2}{d-n+2}.
\]
 We may now consider these rows to be ``removed" from the matrix $\bar{M}_d$ and compute the rank of the remaining submatrix.

    \textbf{Case 2: } $k=2.$

    A unit entry in row $e_{2,\:B}$ arises from a column $e_{1,A}$ with $B=A\setminus\{1\}.$ Hence, the row $e_{2,\:B}$ is zero if $1\in B$. 

    Suppose $1\notin B$. Then $e_{2,\:B}$ has a unit in the column $e_{1,B\cup\{1\}}$. We need to be cautious that the contribution of $e_{2,\:B}$ to the rank hasn't been considered already. 

    Suppose $\{3,4,\dots,n-3,n-1\}\subset B$. Then the column $e_{1,B\cup\{1\}}$ has two nonzero entries (before expansion of invalid rows), in the rows $e_{2,\:B}$ and $e_{n,\:B\cup\{1\}\setminus\{n-1\}}$ respectively. These are precisely the rows already considered in Subcase 2 of Case 1 and thus shouldn't be considered again in the rank computation. 

    So we assume $\{3,4,\dots,n-3,n-1\}\not\subset B$. Then each of the valid rows $e_{2,\:B}$ contributes independently to the rank of $\bar{M}_d$. However, if $e_{2,\:B}$ is invalid, then each of the rows $e_{l,\:B\cup\{2\}\setminus\{l\}}$ has a unit in the column $e_{1,B\cup\{1\}}$. If $l=n$, the row $e_{l,\:B\cup\{2\}\setminus\{l\}}$ is one of the rows ``removed" from the matrix $\bar{M}_d$ at the end of the discussion in Case 1. 

    If $B\neq\{n\}$, then choose $s\in B$ such that $s\neq n$. The row $e_{s,B\cup\{2\}\setminus\{s\}}$ has only one nonzero entry (in the column $e_{1,B\cup\{1\}}$) and is equal (up to multiplication by $\pm 1$) to each of the other rows $e_{l,\:B\cup\{2\}\setminus\{l\}}$ for $l\in B\setminus\{n\}$. Hence in this case, $e_{2,\:B}$ contributes one to the rank as usual.

    However, if $B=\{n\}$, then $e_{2,\{n\}}=e_{n,\{2\}}$. So the contribution of $e_{2,\:B}$ to the rank has already been considered in Case 1.

    Therefore, the contribution of the rows $e_{2,\:B}$ (valid or not) to the rank of $\bar{M}_d$ is the cardinality of the set 
    
\[
\{B\:|\:B\subset[m-1],\: |B|=d-1,\: 1\notin B, \: \{3,4,\ldots,n-3,n-1\}\not\subset B \text{ and } B\neq\{n\}\},
\]
    which equals $m-3$ for $d=2$, and  
    
\[
\binom{m-2}{d-1}-\binom{m-n+2}{d-n+3}
\]
 for $d>2$.
    Thus, $\alpha_2=2m-4$ and 
    \begin{align*}
        \alpha_d &=\binom{m-1}{d-1}-\binom{m-n+2}{d-n+2}+\binom{m-2}{d-1}-\binom{m-n+2}{d-n+3}\\
        &=\binom{m-1}{d-1}+\binom{m-2}{d-1}-\binom{m-n+3}{d-n+3}
    \end{align*}
    for $d>2$.

    For $n=5$, the set $\{1,3,\ldots,n-3,n-1\}$ is replaced by the set $\{1,4\}$. As in the above discussion, this set has cardinality $n-3$ and hence, most of the above computations work identically. However, for $d=2$, the contribution of the rows $e_{2,\:B}$ to the rank of $\bar{M_d}$ is now the cardinality of the set 
\[
\{B\:|\:B\subset[m-1],\: |B|=d-1,\: 1\notin B, \: 4\notin B \text{ and } 5\notin B\},
\]
 which is $m-4$. Thus, $\alpha_2=2m-5$ and as before,
    \begin{align*}
        \alpha_d&=\binom{m-1}{d-1}+\binom{m-2}{d-1}-\binom{m-n+3}{d-n+3}\\
        &=\binom{m-1}{d-1}+\binom{m-2}{d-1}-\binom{m-2}{d-2}\\
        &=2\binom{m-2}{d-1}.
    \end{align*}

    For $n=4$, the key difference is that the set $\{1,3,\ldots,n-3,n-1\}$ is replaced by the set $\{1,3\}$, which has cardinality $n-2$ (and not $n-3$ as before). All the arguments discussed above hold identically. 

    Hence, the $k=n$ case contributes 
\[
\binom{m-1}{d-1}-\binom{m-3}{d-3}
\]
 to the rank, and the $k=2$ case contributes $m-4$ for $d=2$ and 
\[
\binom{m-2}{d-1}-\binom{m-3}{d-2}
\]
 for $d>2$. Thus, $\alpha_2=2m-5$ and 
    \begin{align*}
        \alpha_d&=\binom{m-1}{d-1}-\binom{m-3}{d-3}+\binom{m-2}{d-1}-\binom{m-3}{d-2}\\
        &=\binom{m-1}{d-1}+\binom{m-2}{d-1}-\binom{m-2}{d-2}\\
        &=2\binom{m-2}{d-1}
    \end{align*}
    for $d>2$.
\end{proof}

With the help of the lemma above, we now calculate the Betti numbers of the Ap\'ery toric ideal $J_S$ for $S=\Gamma_m(2, n)$, $n \geq 4$.

\begin{theorem}\label{thm:betti-ideal-2n}
Let $4\leq n<m$ be integers and $S=\Gamma_m(2,n)$. Then, for the Ap\'ery toric ideal $J_S$ of $S$, we have
\begin{enumerate}
    \item[(a)] $\beta_0(T/J_S)=1$.
    
    \item[(b)] If $n=4,5$, then
    \[
    \beta_d(T/J_S)=
    \begin{cases}
    \displaystyle \binom{m}{2}-2m+5, & d=1,\\[2mm]
    \displaystyle
    2\binom{m}{3}-(2m-5)-2\binom{m-2}{2}, & d=2,\\[2mm]
    \displaystyle
    d\binom{m}{d+1}
    -2\binom{m-1}{d},
    & d\ge3.
    \end{cases}
    \]

    \item[(c)] If $n\ge6$, then
    \[
    \beta_d(T/J_S)=
    \begin{cases}
    \displaystyle \binom{m}{2}-2m+4, & d=1,\\[2mm]
    \displaystyle
    2\binom{m}{3}
    -2\binom{m-1}{2}
    -\binom{m-2}{1}
    +\binom{m-n+3}{6-n},
    & d=2,\\[2mm]
\displaystyle d \binom{m}{d+1}- \binom{m}{d}-\binom{m-1}{d}+\binom{m-n+4}{d-n+4}
    & d\ge3.
    \end{cases}
    \]
\end{enumerate}
\end{theorem}
\begin{proof}
We proceed as in the proof of Theorem \ref{thm:betti-ideal}. Firstly, it is clear that $\beta_0(T/J_S)=1$ for all $n \geq 4$. Since the complex $(\FF_\bullet,\partial_\bullet)$ from Construction~\ref{cons:ideal-res} is a free resolution of the Ap\'ery toric quotient $T/J_S$, by  Lemma \ref{lem:alpha-gives-betti-nos}, given any $d \geq 1$, we have
\[
\beta_d(T/J_S)=\rank(F_d)-\alpha_d(\FF_\bullet)-\alpha_{d+1}(\FF_\bullet),
\]
where $\alpha_d(\FF_\bullet)=\rank(\overline{M}_d)$ is given by Lemma~\ref{lem:alphad-2n}, and by Construction \ref{cons:ideal-res}, we have \(\operatorname{rank}(F_d) = d\binom{m}{d+1}\). 

\textbf{Case 1}: $n=4,5$ \\ In this case, 
\[
 \beta_1(T/J_S)=\rank(F_1)- \alpha_1(\FF_\bullet) -\alpha_{2}(\FF_\bullet)=\binom{m}{2}-2m+5,
\]
 
\[
\beta_2(T/J_S)=\rank(F_2)- \alpha_2(\FF_\bullet) -\alpha_{3}(\FF_\bullet)=2\binom{m}{3}-(2m-5)-2 \binom{m-2}{2},
\]
 and for $d \geq 3$ we have
\[
\beta_d(T/J_S)=\rank(F_d)- \alpha_d(\FF_\bullet) -\alpha_{d+1}(\FF_\bullet)=d \binom{m}{d+1}-2 \binom{m-2}{d-1}-2 \binom{m-2}{d}= d\binom{m}{d+1}
    -2\binom{m-1}{d}.
\]

\textbf{Case 2}: $n\geq 6$ \\ In this case, 
\[
 \beta_1(T/J_S)=\rank(F_1)- \alpha_1(\FF_\bullet) -\alpha_{2}(\FF_\bullet)=\binom{m}{2}-2m+4,
\]
 
\begin{align*}
 \beta_2(T/J_S)=& \rank(F_2)- \alpha_2(\FF_\bullet) -\alpha_{3}(\FF_\bullet)\\
 &=2\binom{m}{3}-(2m-4)-\binom{m-2}{2}-\binom{m-1}{2}+\binom{m-n+3}{6-n} 
 \\
 &= 2\binom{m}{3}
-2\binom{m-2}{2}
-3\binom{m-2}{1}
+\binom{m-n+3}{6-n}  \\
&=  2\binom{m}{3}
    -2\binom{m-1}{2}
    -\binom{m-2}{1}
    +\binom{m-n+3}{6-n},
\end{align*}
and for $d \geq 3$ we have
\begin{align*}
  \beta_d(T/J_S)&=\rank(F_d)- \alpha_d(\FF_\bullet) -\alpha_{d+1}(\FF_\bullet) \\
  &=d \binom{m}{d+1}- \left[ \binom{m-2}{d-1}+\binom{m-1}{d-1} - \binom{m-n+3}{d-n+3} \right] - \left[ \binom{m-2}{d}+\binom{m-1}{d} - \binom{m-n+3}{d-n+4} \right] \\
  &= d\binom{m}{d+1}
-2\binom{m-2}{d}
-3\binom{m-2}{d-1}
-\binom{m-2}{d-2}
+\binom{m-n+3}{d-n+3}
+\binom{m-n+3}{d-n+4} \\
&=d \binom{m}{d+1}- \binom{m}{d}-\binom{m-1}{d}+\binom{m-n+4}{d-n+4}.
\qedhere
\end{align*}

\end{proof}

We are now ready to calculate the Betti numbers of the defining toric ideal of $\Gamma_m(2,n)$. We split the calculation into two separate theorems, corresponding to the two cases in the preceding theorem.
\begin{theorem}\label{thm:defining-ideal-for-24-25}
Let $n \in \{4,5\}$, $m>n$, and $S=\Gamma_m(2,n)$.  Then $\beta_0(Q/I_S)=1$, and for every $d\ge1$,
\[
\beta_d(Q/I_S)=d\binom{m-2}{d+1}.
\]
\end{theorem}

\begin{proof}
 Clearly $\beta_0(Q/I_S)=1$. By Proposition \ref{prop:tensorwithKoszul}, we have
\begin{equation}\label{eq:gf-n4}
\sum_{i\ge0}\beta_i(T/J_S)z^i=\Big(\sum_{i\ge0}\beta_i(Q/I_S)z^i\Big)(1+z)^2.
\end{equation}
Using Theorem~\ref{thm:betti-ideal-2n}, the left-hand side of the above equation is given by

\begin{align*}
\sum_{d\ge0}\beta_d(T/J_S)\,z^d
&= 1 + \left(\binom{m}{2}-2m+5\right)z + \left(2\binom{m}{3}-(2m-5)-2\binom{m-2}{2}\right)z^2 \\& \quad
   + \sum_{d\ge3}\left(d\binom{m}{d+1}-2\binom{m-1}{d}\right)z^d\\[2mm]
&= 3+3z+z^2
   + \left[-2+\left(\binom{m}{2}-2m+2\right)z+\left(2\binom{m}{3}-2(m-2)-2\binom{m-2}{2}\right)z^2\right]\\
&\qquad + \sum_{d\ge3}\left(d\binom{m}{d+1}-2\binom{m-1}{d}\right)z^d\\[2mm]
&= 3+3z+z^2 + \sum_{d\ge0}\left(d\binom{m}{d+1}-2\binom{m-1}{d}\right)z^d\\[2mm]
&= 3+3z+z^2 + \sum_{d\ge0} d\binom{m}{d+1}z^d
   - 2\sum_{d\ge0}\binom{m-1}{d}z^d\\[2mm]
&= 3+3z+z^2 + \sum_{d\ge0} d\binom{m}{d+1}z^d - 2(1+z)^{m-1}\\[2mm]
&= 3+3z+z^2 + \left(m(1+z)^{m-1}-\frac{(1+z)^m-1}{z}\right) - 2(1+z)^{m-1}\\
&= 3+3z+z^2 + (m-2)(1+z)^{m-1}-\frac{(1+z)^m-1}{z}.
\end{align*}
In view of Equation \eqref{eq:gf-n4}, dividing the above equation by $(1+z)^2$, we get
\begin{align*}
\sum_{d\ge0}\beta_d(Q/I_S)\,z^d
&= \frac{1}{(1+z)^2}\sum_{d\ge0}\beta_d(T/J_S)\,z^d\\[2mm]
&= \frac{1}{(1+z)^2}\left[(3+3z+ z^2)+(m-2)(1+z)^{m-1}-\frac{(1+z)^m-1}{z}\right]\\[2mm]
&= \frac{(1+z)^2+z+2}{(1+z)^2}+ \frac{(m-2)(1+z)^{m-1}}{(1+z)^2}-\frac{(1+z)^m-1}{z(1+z)^2}\\[2mm]
&=1+\frac{z}{(1+z)^2}+\frac{2}{(1+z)^2}+(m-2)(1+z)^{m-3}-\frac{(1+z)^{m-2}-(1+z)^{-2}}{z}\\[2mm]
&= 1+\sum_{d\ge0}(-1)^d(d+2)z^d + \sum_{d\ge0}(m-2)\binom{m-3}{d}z^d
   -\sum_{d\ge0}\left[\binom{m-2}{d+1}+(-1)^d(d+2)\right]z^d
   \,\\[2mm]
&= 1+\sum_{d\ge0}\left[(m-2)\binom{m-3}{d}-\binom{m-2}{d+1}\right]z^d\\[2mm]
&= 1+\sum_{d\ge0}\left[(d+1)\binom{m-2}{d+1}-\binom{m-2}{d+1}\right]z^d\\[2mm]
&= 1+\sum_{d\ge0} d\binom{m-2}{d+1}\,z^d \\
&= 1+\sum_{d\ge 1} d\binom{m-2}{d+1}\,z^d.
\end{align*}
This completes the proof.
\end{proof}

We now compute the Betti numbers of the defining toric ideal of $\Gamma_m(2,n)$ for $n \geq 6$.

\begin{theorem}\label{thm:defining-ideal-for-2n}
Let $6\le n< m$ and  $S=\Gamma_m(2,n)$. Then $\beta_0(Q/I_S)=1$, and for every $d\ge1$,
\[
\beta_d(Q/I_S)=
d\binom{m-2}{d+1}
-\binom{m-3}{d-1}
+\binom{m-n+2}{d-n+4}.
\]
\end{theorem}
\begin{proof}
Clearly $\beta_0(Q/I_S)=1$. By Proposition \ref{prop:tensorwithKoszul}, we have
\begin{equation}\label{eq:gf-ngeq5}
\sum_{i\ge0}\beta_i(T/J_S)z^i=\Big(\sum_{i\ge0}\beta_i(Q/I_S)z^i\Big)(1+z)^2.
\end{equation}
Using Theorem~\ref{thm:betti-ideal-2n}, for $n\ge6$, the left-hand side of the above equation is given by

\begin{align*}
\sum_{d\ge0}\beta_d(T/J_S)\,z^d
&= 1 + \left(\binom{m}{2}-2m+4\right)z + \left(2\binom{m}{3}
    -2\binom{m-1}{2}
    -\binom{m-2}{1}
    +\binom{m-n+3}{6-n}\right)z^2\\
&\quad + \sum_{d\ge3}\left(d \binom{m}{d+1}- \binom{m}{d}-\binom{m-1}{d}+\binom{m-n+4}{d-n+4}\right)z^d\\
&=3+3z+z^2
+\sum_{d\ge0}\left(d \binom{m}{d+1}- \binom{m}{d}-\binom{m-1}{d}+\binom{m-n+4}{d-n+4}\right)z^d\\[2mm]
&=3+3z+z^2 + \sum_{d\ge0}d\binom{m}{d+1}z^d - (1+z)^m - (1+z)^{m-1} + z^{n-4} (1+z)^{m-n+4}\\
&= 3+3z+z^2 + \sum_{d\ge0}d\binom{m}{d+1}z^d - (2+z) (1+z)^{m-1} + z^{n-4} (1+z)^{m-n+4} \\[2mm]
&=3+3z+z^2
+\left(m(1+z)^{m-1}
-\frac{(1+z)^m-1}{z}\right)
-(2+z)(1+z)^{m-1}
+z^{\,n-4}(1+z)^{m-n+4}\\[2mm]
&=3+3z+z^2+(m-2-z)(1+z)^{m-1}
-\frac{(1+z)^m-1}{z}
+z^{\,n-4}(1+z)^{m-n+4}.
\end{align*}
In view of Equation \eqref{eq:gf-ngeq5}, dividing the above equation by $(1+z)^2$, we get
\begin{align*}
\sum_{d\ge0}\beta_d(Q/I_S)\,z^d
&=\frac{1}{(1+z)^2}\sum_{d\ge0}\beta_d(T/J_S)\,z^d\\[2mm]
&=\frac{1}{(1+z)^2}\left[3+3z+z^2+(m-2-z)(1+z)^{m-1}
-\frac{(1+z)^m-1}{z}
+z^{\,n-4}(1+z)^{m-n+4}\right]\\[2mm]
&=  1+ \dfrac{z+2}{(1+z)^2}+(m-2-z)(1+z)^{m-3}-\frac{(1+z)^{m-2}-(1+z)^{-2}}{z} + z^{n-4}(1+z)^{m-n+2} \\[2mm]
&=1+\sum_{d\ge0}(-1)^d(d+2)z^d+\sum_{d\ge0}\Big[(m-2)\binom{m-3}{d}-\binom{m-3}{d-1}\Big]z^d
- \\ &\qquad \sum_{d\ge0}\Big[\binom{m-2}{d+1}+(-1)^d(d+2)\Big]z^d
+\sum_{d\ge0}\binom{m-n+2}{d-n+4}z^d \\[2mm]
&=1+\sum_{d\ge 0}\left[(m-2)\binom{m-3}{d}
-\binom{m-3}{d-1}
-\binom{m-2}{d+1}
+\binom{m-n+2}{d-n+4}\right]z^d\\[2mm]
&=1+\sum_{d\ge0}\left[
d\binom{m-2}{d+1}
-\binom{m-3}{d-1}
+\binom{m-n+2}{d-n+4}
\right]z^d\\[2mm]
&=1+\sum_{d\ge1}\left[
d\binom{m-2}{d+1}
-\binom{m-3}{d-1}
+\binom{m-n+2}{d-n+4}
\right]z^d,
\end{align*}
where we used the identity
$(m-2)\binom{m-3}{d}=(d+1)\binom{m-2}{d+1}$.
\end{proof}

\section{Betti numbers of the residue field}
\label{sec:residue-field-gamma-n}

In this section, we compute the Betti numbers of the residue field $\KK$ over the semigroup ring corresponding to the Sally type numerical semigroups $\Gamma_m(n)$, where $3\leq n<m$. Let $S$, $R$, $T$, $Q$, $I_S$, and $J_S$ be as in Section \ref{sec:definingideals}. We begin by recalling the infinite Ap\'ery resolution of $\KK$ over $R=T/J_S$ constructed in \cite{gosd24}.  

\begin{construction}[{\cite[Definition 3.1]{gosd24}}]\label{cons:k-res}
    For $d\geq 1$, let $G_d$ be the free $R$-module with basis
    \[
    \{e_{\mathbf v}:\mathbf v=(v_1,\ldots,v_d)\in(\ZZ/m\ZZ)^d,
    \quad v_i\neq 0\text{ for }i>1\},
    \]
    and set $G_0=Re_{\emptyset}$. In particular,
    \begin{equation}
    \label{eq:residue-free-ranks}
    \rank(G_0)=1,
    \qquad
    \rank(G_d)=m(m-1)^{d-1}\quad\text{for }d\geq 1.
    \end{equation}
    For $\mathbf v=(v_1,\ldots,v_d),$ we write $\widehat{\mathbf v}=(v_1,\ldots,v_{d-1})$
    and define
    \[
    \tau_i\mathbf v
    =(v_1,\ldots,v_{i-1},v_i+v_{i+1},v_{i+2},\ldots,v_d),
    \]
    where the entries are interpreted modulo $m$. We use the convention that $e_{\mathbf u}=0$ whenever some
    entry $u_j$ with $j>1$ is zero. The differential $\partial_d:G_d\to G_{d-1}$ is
    \begin{equation}
    \label{eq:residue-apery-differential}
    \partial_d(e_{\mathbf v})
    =x_{v_d}e_{\widehat{\mathbf v}}
    +\sum_{i=1}^{d-1}(-1)^{d-i}y^{c_{v_i,v_{i+1}}}
    e_{\tau_i\mathbf v}.
    \end{equation}
    \end{construction}
    
\begin{theorem}[{\cite[Theorem 3.2]{gosd24}}]
    The complex
    \[
    \GG_\bullet:\quad
    0\longleftarrow G_0\longleftarrow G_1\longleftarrow G_2
    \longleftarrow\cdots
    \]
    is a free resolution of $\KK$ over $R$. This resolution is minimal iff $S$ has maximal embedding dimension. \qed
\end{theorem}

Let $S=\Gamma_m(n)$ for $3\leq n<m$, and recall that for such an $S$, we have that $c_{ij}=0$ if and only if $i+j=n$ with $1 \leq i , j \leq n-1$. Let $M_d$ denote the matrix representing the differential $\partial_d$ in the Ap\'ery resolution $\GG_\bullet$ of the residue field $\KK$ over $T/J_S$, and let $\overline{M_d}$ denote its
reduction modulo $\mathfrak m$. We use a similar strategy as in the previous section, and make use of Lemma \ref{lem:alpha-gives-betti-nos} to extract the Betti numbers of $\KK$. 

We label the rows of $\overline{M_d}$ by vectors $\mathbf v=(v_1,\ldots,v_{d-1})$, where $v_1 \in \ZZ / m \ZZ$ and $v_j\in \ZZ / m \ZZ \setminus \{0\}$ for all $j \geq 2$, and call such vectors \emph{admissible}. If $v_p=n$, then we let $\mathbf v^{p,i,n-i}$ denote the word obtained by replacing the entry $v_p=n$ by the ordered pair $(i,n-i)$. For $d\geq 1$ and $i\geq 1$, let $A_{d,i}$ denote the number of rows of $M_d$ whose labels contain at least $i$ occurrences of $n$, that is, the rows of $M_d$ which contain at least $i$ \emph{label occurrences} of $n$. We set $A_{q,i}=0$ for $q\leq 0$.

Recall that $\alpha_d(\GG_\bullet)$ is defined to be the rank of $\overline{M_d}$. 

\begin{theorem}
\label{thm:residue-rank-formula}
For every $d\geq 1$,
\[
\alpha_d(\GG_\bullet)
=\rank(\overline{M_d})
=A_{d,1}-A_{d-1,2}+A_{d-2,3}-A_{d-3,4}+\cdots.
\]
\end{theorem}
\begin{proof}
By \eqref{eq:residue-apery-differential}, every row of $\overline{M_d}$ whose label does not contain the symbol $n$ is zero. Therefore, we need to find the dimension of the span of the $A_{d,1}$ rows containing at least one label occurrence of $n$. Such a row of $\overline{M_d}$ indexed by $\bv$ has unit entries in the columns indexed by 
\[
\{\bv^{p,i,n-i}\;|\;p \text{ such that } v_p=n, \;i=1,\ldots,n-1\}.
\] Moreover, the unit entry in the column indexed by $\bv^{p,i,n-i}$ is $(-1)^{d-p}$. Thus, we say that the row vector indexed by $\bv$ is 
\[
r_{\mathbf v}
\colonequals\sum_{\substack{p \\ v_p=n}}(-1)^{d-p}
\sum_{i=1}^{n-1}{\mathbf v^{p,i,n-i}}
\]
We now define the following complex $\WW_\bullet$ of $\KK$-vector spaces: given $\ell\geq 1$ and $t\geq 0$, let $W_{\ell,t}$ denote the $\KK$-vector space whose basis elements are admissible vectors
$\bw=(w_1,\ldots,w_{\ell-1})$ that contain \emph{exactly} $t$ occurrences of the symbol $n$. For $t\geq1$, define $
\delta_{\ell,t}:W_{\ell,t}\longrightarrow W_{\ell+1,t-1}$ as
\begin{equation}
\label{eq:residue-word-differential}
\delta_{\ell,t}(\bw)
=\sum_{\substack{p \\ w_p=n}}(-1)^{\ell-p}
\sum_{i=1}^{n-1}\mathbf w^{p,i,n-i}.
\end{equation}
A simple check yields $\delta^2=0$.
Observe that for $\bv$ as above, $\delta(\bv)=r_\bv$. Thus, any element in the kernel of $\delta: W_{d,t}\to W_{d+1,t-1}$ produces a relation between the rows of $\overline{M_d}$. Conversely, any relation between the rows of $\overline{M_d}$ can be separated into relations which involve only rows indexed by vectors with the same number of label occurrences of $n$. Each of these relations can be represented by an element of $\ker(\delta)$. 

Let $D_{\ell,i}=\bigoplus_{t\geq i}\delta_{\ell,t}$. By Lemma \ref{lem:residue-word-exactness}, the complex $\WW$ is exact. Since $\delta_{\ell,t}$ reduces the number of label occurrences of $n$ by exactly one, we get 
\[
\ker\left(
D_{\ell,i}:\bigoplus_{t\geq i}W_{\ell,t}
\longrightarrow
\bigoplus_{t\geq i-1}W_{\ell+1,t}
\right)
=
D_{\ell-1,i+1}\left(\bigoplus_{t\geq i+1}W_{\ell-1,t}\right).
\]
Thus, every relation among rows of $\overline{M_d}$ whose labels contain at least $i$ occurrences of $n$ is a linear combination of relations arising from rows of $\overline{M_{d-1}}$ whose labels contain at least $i+1$ occurrences of $n$.

We know that there are $A_{d,1}$ rows with at least one label occurrence of $n$.  Now, a linear relation among these rows is an element of the kernel of the corresponding word map $D_{d,1}$. As $\WW$ is exact, the space of all such relations is precisely the image of the space spanned by the rows of $M_{d-1}$ containing at least two label occurrences of $n$. By the Rank-Nullity theorem, 
\[
\alpha_d=A_{d,1}-\rank\left(D_{d-1,2}\right)
\]

Once again, as $\WW$ is exact, the kernel of $D_{d-1,2}$ is the image of the rows of $M_{d-2}$ containing at least three label occurrences of $n$. Repeated application of the Rank-Nullity theorem gives the alternating sum as in the statement, and hence the theorem follows.
\end{proof}

A key component of the above proof was the exactness of the complex $\WW$, which we now prove.
Given an admissible word $\mathbf v$, we define its \emph{full expansion}
$E(\mathbf v)$ to be the word obtained by replacing every occurrence of the symbol $n$ by the two-letter word $1(n-1)$. Also, we let $\vert\bv\vert_n$ denote the number of occurrences of $n$ in $\bv$.

Before moving to the next result about the complex $\mathbb W$, in the next remark we recall a well-known fact about simplicial chain complexes and their homology. For more details, we refer the reader to \cite[Chapter 2]{Hatcher}.

\begin{remark}\label{rem:homology-rem}
Let $\Delta$ be a nonempty simplicial complex. Its augmented
simplicial chain complex over $\KK$ is
\[
\cdots\to C_2(\Delta;\KK)
\xrightarrow{\partial_2}C_1(\Delta;\KK)
\xrightarrow{\partial_1}C_0(\Delta;\KK)
\xrightarrow{\varepsilon}\KK\to0,
\]
where $C_r(\Delta;\KK)$ is the $\KK$-vector space with basis
the oriented $r$-simplices, with one orientation chosen for
each simplex. For $r\geq1$, the boundary maps are given by
\[
\partial_r[v_0,\ldots,v_r]
=\sum_{i=0}^r(-1)^i
[v_0,\ldots,\widehat{v_i},\ldots,v_r],
\]
and $\varepsilon([v])=1$ for every vertex $v$.
The homology of this complex is the reduced simplicial
homology of $\Delta$. If $\Delta$ is a nonempty simplex,
then it is contractible, so its augmented simplicial
chain complex is exact.
\end{remark}

\begin{lemma} \label{lem:residue-word-exactness}
Let $\WW$ be the complex of $\KK$-vector spaces defined above. For every $\ell\geq 2$ and $t\geq 1$, we have $\ker\left(\delta:W_{\ell,t}\longrightarrow W_{\ell+1,t-1}\right) =
\im\left(\delta:W_{\ell-1,t+1}\longrightarrow W_{\ell,t}\right)$. 
\end{lemma}
\begin{proof}
    Observe that for every admissible word $\mathbf{v}$, the quantity $\vert E(\mathbf v)\vert=\vert\mathbf v\vert+\vert\mathbf v\vert_n$ is preserved by $\delta$. Consider any lexicographic order $<$ on the set of expanded words in which we have $1 >j$ for every $j \neq 1$. Then we see that $E(\mathbf v^{p,1,n-1})=E(\mathbf v)$, and
$E(\mathbf v^{p,i,n-i})<E(\mathbf v)$
for all $2\leq i\leq n-1$.

 The above lexicographical order thus induces a finite filtration on the subcomplex of each fixed expanded length. Its associated graded differential retains only the
terms that preserve the full expansion, and is therefore given by
\begin{equation}
\label{eq:residue-graded-differential}
\delta_0(\mathbf v)
=\sum_{\substack{p\\ v_p=n}}(-1)^{\ell-p}\mathbf v^{p,1,n-1}.
\end{equation}

Now, fix a word $\mathbf u$ having no occurrence of $n$, and let
\[
P(\mathbf u)=\{p \mid u_p=1\text{ and }u_{p+1}=n-1\}.
\]
Since $n\geq 3$, the symbols $1$ and $n-1$ are distinct, and hence no two occurrences of the subword $1(n-1)$ in $\mathbf u$ overlap. Every word $\mathbf v$ satisfying $E(\mathbf v)=\mathbf u$ is therefore obtained uniquely by choosing a subset $Q\subseteq P(\mathbf u)$ and replacing each occurrence of $1(n-1)$ indexed by $Q$ with the symbol $n$.

Let $\Delta_{\mathbf u}$ be the simplex with vertex set $P(\mathbf u)$. Under this identification, the basis elements are indexed by the faces $Q\subseteq P(\mathbf u)$ of $\Delta_{\mathbf u}$. The differential $\delta_0$ induces the map
\begin{equation}
\label{eq:residue-simplex-differential}
Q\longmapsto \sum_{p\in Q}\varepsilon(Q,p)\bigl(Q\setminus\{p\}\bigr),
\end{equation}
where $\varepsilon(Q,p)\in\{1,-1\}$. 
After rescaling the basis elements by appropriate signs,
\eqref{eq:residue-simplex-differential} agrees with the
augmented simplicial boundary map of $\Delta_{\mathbf u}$,
with word degree $t$ corresponding to simplicial degree
$t-1$. Thus, when $P(\mathbf u)$ is nonempty, this complex
is exact by Remark~\ref{rem:homology-rem}.
If $P(\mathbf u)=\emptyset$, then the complex is concentrated in degree zero, with the degree zero component generated by the word $\mathbf u$. Hence, in this case as well, the complex has no homology in positive degrees.

Now, suppose that $t \geq 1$ and $x = \sum a_j \mathbf v_j \in W_{\ell,t}$ is such that 
$\delta(x)=0$. Let $\mathbf u$ be the largest full expansion of the terms $\mathbf v_j$ appearing in
$x$, and let $x_{\mathbf u}$ denote the part of $x$ with full expansion
$\mathbf u$, i.e., $x_{\mathbf u} = \sum\limits_{\substack{j \\ E(\mathbf v_j) = \mathbf u}} a_j \mathbf v_j$. Since every term coming from a splitting other than
$n=1+(n-1)$ has smaller full expansion, from
$\delta(x)=0$, we get $
\delta_0(x_{\mathbf u})=0$.
Since $t\geq1$, we have $P(\mathbf u)\neq\emptyset$.
The exactness of the augmented simplicial chain complex therefore gives an element $y_{\mathbf u}\in W_{\ell-1,t+1}$ such that $x_{\mathbf u}=\delta_0(y_{\mathbf u})$.

Thus, we see that in the full expansion of $x-\delta(y_{\mathbf u})$, every term is smaller than $\mathbf u$. We may now repeat the above process with $x$ replaced by $x - \delta(y_{\mathbf u})$. Since there are only finitely many words of a fixed expanded length, this process terminates after finitely many steps and yields  $x \in \im(\delta)$. This completes the proof.
\end{proof}

For $d\geq 2$, let $N_{d,t}$ denote the number of rows of $M_d$ whose labels
contain exactly $t$ occurrences of $n$. The first entry of a row label has $m$
possible values, whereas each later entry has $m-1$ possible values. By
separating the cases in which the first entry is $n$ or not, we get  
\[
N_{d,t}
=\binom{d-2}{t-1}(m-2)^{d-1-t}
+(m-1)\binom{d-2}{t}(m-2)^{d-2-t}.
\]
Note that we have $A_{d,i}=\sum_{t=i}^{d-1}N_{d,t}$. This gives
\[
\begin{aligned}
\sum_{d\geq 1}\alpha_d(\GG_\bullet)z^d
&=\sum_{d\geq2}\sum_{t\geq1}N_{d,t}z^d
-\sum_{d\geq2}\sum_{t\geq2}N_{d,t}z^{d+1}
+\sum_{d\geq2}\sum_{t\geq3}N_{d,t}z^{d+2}-\cdots\\
&=\sum_{d\geq2}z^d\big(N_{d,1}+N_{d,2}(1-z)+N_{d,3}(1-z+z^2)+\cdots\big)\\
&=\frac{1}{1+z}
\sum_{d\geq2}\sum_{t\geq1}
N_{d,t}z^d\bigl(1-(-z)^t\bigr).
\end{aligned}
\]

Using the value of $N_{d,t}$ and applying the binomial theorem, we get
\[
\begin{aligned}
\sum_{t\geq1}N_{d,t}\bigl(1-(-z)^t\bigr)
={}&m(m-1)^{d-2}-(m-1-z)(m-2-z)^{d-2}.
\end{aligned}
\]

Therefore, we obtain
\begin{align*}
\sum_{d\geq1}\alpha_d(\GG_\bullet)z^d
&=\frac{1}{1+z}\sum_{d\geq2}z^d
\Bigl[
m(m-1)^{d-2}-(m-1-z)(m-2-z)^{d-2}
\Bigr] \\
&= \frac{z^2}{1+z}\left[
\frac{m}{1-(m-1)z}-\frac{m-1-z}{1-(m-2)z+z^2}
\right] \\ &= 
\frac{z^2(1+z)}
{\bigl(1-(m-1)z\bigr)
\bigl(1-(m-2)z+z^2\bigr)}.
\end{align*}

The Betti numbers of $\KK$ can now be obtained simply by using Lemma~\ref{lem:alpha-gives-betti-nos}.

\begin{theorem} \label{thm:residue-betti-gamma-n}
    Let $3 \leq n <m$ and $S=\Gamma_m(n)$, and let $R=\KK[S]$. Then the Poincar\'e series $\mathcal P_{\KK}^R(z) = \sum_{i\geq 0}\beta_i^R(\KK)z^i$ of the residue field $\KK$ is given by 
    \[
    \mathcal P_{\KK}^R(z) 
    =\frac{1+z}{1-(m-2)z+z^2}.
    \]
    Equivalently,
    \[
    \beta_0^R(\KK)=1,
    \qquad
    \beta_1^R(\KK)=m-1,
    \]
    and
    \[
    \beta_d^R(\KK)
    =(m-2)\beta_{d-1}^R(\KK)-\beta_{d-2}^R(\KK)
    \qquad\text{for }d\geq 2.
    \]
\end{theorem}
\begin{proof}
    The recurrence follows simply by expanding the rational function expression of $\mathcal P_{\KK}^R(z)$ obtained as a power series. 
    
    We have $\sum_{d\geq 0}\rank(G_d)z^d = 1+ \sum_{d\geq 1} m(m-1)^{d-1}z^d =  \frac{1+z}{1-(m-1)z}$. To get the required expression of $\mathcal P_{\KK}^R(z) $, observe that by Lemma~\ref{lem:alpha-gives-betti-nos} and the expression for the power series of $\alpha_d(\GG_{\bullet})$ obtained earlier, we have
    \begin{align*}
        \mathcal P_{\KK}^R(z) 
        &=
        \sum_{d\geq 0}\rank(G_d)z^d - \sum_{d\geq 1}\alpha_d(\GG_\bullet)z^d - \dfrac{1}{z}\sum_{d\geq 1}\alpha_d(\GG_\bullet)z^d\\
        &= \frac{1+z}{1-(m-1)z} - \frac{z^2(1+z)}
{\bigl(1-(m-1)z\bigr)
\bigl(1-(m-2)z+z^2\bigr)} - \frac{z(1+z)}
{\bigl(1-(m-1)z\bigr)
\bigl(1-(m-2)z+z^2\bigr)} \\
&= \frac{1+z}{1-(m-2)z+z^2}.
    \end{align*} 
    This completes the proof.
\end{proof}

\begin{remark}
In particular, for a fixed multiplicity $m$, the Poincar\'e series
obtained in the theorem above is independent of the parameter $n$ as
long as $3\leq n<m$.
\end{remark}

\begin{theorem}[cf. {\cite[Subsection 4.2]{gosd24}}]\label{thm:Poincare-series-for-n=2}
    Let $m\geq4$, $S=\Gamma_m(2)$, and $R=\KK[S]$. Then the Poincar\'e series $\mathcal{P}_{\KK}^R(z)$ of the residue field $\KK$ is given by
    \[
    \mathcal{P}_{\KK}^R(z)=\frac{1+z}{1-(m-2)z}.
    \]
    Equivalently, $\beta_0^R(\KK)=1$ and $\beta_d^R(\KK)=(m-1)(m-2)^{d-1}$ for $d\geq 1$.
\end{theorem}

\begin{remark}
    Although Subsection 4.2 of \cite{gosd24} assumes $m=4$, all the arguments follow through verbatim for larger $m$. 
\end{remark}

\begin{defn}
    Let $Q=\KK[x_1,\ldots,x_e]$ be non-negatively graded, $I$ be a homogeneous ideal of $Q$ such that $I\subset(x_1,\ldots,x_e)^2$, and $R=Q/I$. The graded algebra $R$ is called \emph{Golod} if 
    \[
    \mathcal{P}_\KK^R(z)=\frac{(1+z)^e}{1-z(\mathcal{P}^Q_R(z)-1)}.
    \]
\end{defn}

\begin{defn}\label{defn:Goloddefect}
    Let $Q$, $I$, and $R$ be as above. Let 
    \[
    \sum_{i\geq0}g_iz^i\colonequals\frac{(1+z)^e}{1-z(\mathcal{P}^Q_R(z)-1)}.
    \]
    Serre proved that $g_i\geq\beta_i^R(\KK)$ for all $i$. We define the $i^{\text{th}}$ \emph{Golod defect number} of $R$ to be $\mathcal{D}_i(R)\colonequals g_i-\beta_i^R(\KK)$, and the \emph{Golod defect index} of $R$ to be 
    \[
    \gdi(R)\colonequals\inf\big(\{i\mid \mathcal{D}_i(R)>0\}\cup\{\infty\}\big).
    \]
\end{defn}

\begin{remark}
    The invariant $\mathcal{D}_i(R)$ has been studied previously in \cite{NguyenVeliche}. We give it the terminology of Golod defect number. 
    
    As the name suggests, the Golod defect numbers and Golod defect index measure how far a ring $R$ is from being Golod, that is, how far and at what stage the Betti numbers of $\KK$ over $R$ fall below Serre's upper bound. If $\gdi(R)\geq N$, then the Betti numbers $\beta_i^R(\KK)$ attain Serre's upper bound for $i<N$. In particular, $R$ is Golod iff $\gdi(R)=\infty$. For more comments on these invariants, see Remark \ref{rmk:gdi}.
\end{remark}

\begin{theorem}\label{thm:golod}
    Let $m\geq4$, $2\leq n<m$, $S=\Gamma_m(n)$, and $R=\KK[S]$. Then $R$ is Golod iff $n=2$. Moreover, if $n\geq3$, then $\gdi(R)=n$ and $\mathcal{D}_n(R)=1$.
\end{theorem}
\begin{proof}
    The minimal number of generators of $S$ is $m-1$. 
    By Theorem \ref{thm:defining-ideal-for-n}, for $2\leq n<m$
    \begin{align*}
        \sum_{i\geq0}g_iz^i
        &\colonequals\frac{(1+z)^{m-1}}{1-z(\mathcal{P}^Q_R(z)-1)}\\
        &=\frac{(1+z)^{m-1}}{1-z\big(m(1+z)^{m-2}-\frac{(1+z)^m-1}{z}+z^{n-1}(1+z)^{m-n}\big)}\\
        &=\frac{(1+z)^{m-1}}{(1+z)^{m-2}\big(1-(m-2)z+z^2\big)-(1+z)^{m-n}z^n}\\
        &=\frac{(1+z)^{n-1}}{(1+z)^{n-2}\big(1-(m-2)z+z^2\big)-z^n} \quad \text{as }n\geq2.
    \end{align*}
    For $n=2$, the above expression readily simplifies to the expression of $\mathcal{P}_\KK^R(z)$ in Theorem \ref{thm:Poincare-series-for-n=2}, thus proving that $R$ is Golod.
    For $n\geq 3$, by Theorem \ref{thm:residue-betti-gamma-n},
    \begin{align*}
        \sum_{i\geq0}g_iz^i-\mathcal{P}_\KK^R(z)
        &=\frac{(1+z)^{n-1}}{(1+z)^{n-2}(1-(m-2)z+z^2)-z^n}-\frac{1+z}{1-(m-2)z+z^2}\\
        &=\frac{z^n(1+z)}{\bigg((1+z)^{n-2}\big(1-(m-2)z+z^2\big)-z^n\bigg)\big(1-(m-2)z+z^2\big)}.
    \end{align*}
    As $z$ does not divide the denominator in the above expression, the resulting power series has order $n$. Hence, $\mathcal{D}_i(R)=0$ for $i<n$. However, $\mathcal{D}_n(R)=1$ and hence, $R$ is not Golod for $n\geq3$.
\end{proof}

\section{Applications}\label{sec:applications}

In this section, we apply our results to two further families of numerical semigroups. First, we obtain families with the same Betti numbers by considering numerical semigroups in the relative interior of the same face of the Kunz cone. We then consider gluings with $\NN$, for which both Poincar\'e series can be expressed in terms of those of the original semigroup ring. 

\subsection{Numerical semigroups in the same face of the Kunz cone}

We briefly recall the description of the Kunz cone in terms of Ap\'ery sets; see \cite{bgmns25,gosd24}. For $m\geq2$, the \emph{Kunz cone} is the rational polyhedral cone
\[
\mathcal{C}_m\colonequals
\{(a_1,\ldots,a_{m-1})\in\RR_{\geq0}^{m-1}\ \mid 
a_i+a_j\geq a_{i+j}\text{ whenever }1\leq i,j\leq m-1\text{ and }i+j\not\equiv0\pmod m\},
\]
where the sum in the subscript is taken modulo $m$. If $S$ has multiplicity $m$ and $\operatorname{Ap}(S)=\{0,a_1,\ldots,a_{m-1}\}$, then $(a_1,\ldots,a_{m-1})\in\mathcal{C}_m$. Conversely, an integer point of $\mathcal{C}_m$ with $a_i>m$ and $a_i\equiv i\pmod m$ is the Ap\'ery vector of a numerical semigroup of multiplicity $m$. Two such vectors lie in the relative interior of the same face precisely when they satisfy the same equalities $a_i+a_j=a_{i+j}$. 

We say that $S$ is on the face of the Kunz cone if its Ap\'ery vector is on the said face.  Kunz \cite{Kunz} proved that the Betti numbers of the defining toric ideals are the same for numerical semigroups in the relative interior of each face. The same is true of the Betti numbers of the residue field by \cite[Proposition 3.5]{gosd24}. 

\begin{proposition}\label{prop:shifted-sally-face}
Let $H=\{n\}$ with $2\leq n<m$, or $H=\{2,n\}$ with $4\leq n<m$. For each integer $q\geq1$, set
\[
S^{(q)}\colonequals\langle m,\ qm+i\mid 1\leq i\leq m-1,\ i\notin H\rangle.
\]
Then $S^{(q)}$ has multiplicity $m$ and embedding dimension $m-|H|$, and its Ap\'ery set is $\{0,a_1^{(q)},\ldots,a_{m-1}^{(q)}\}$, where
\[
a_i^{(q)}=
\begin{cases}
qm+i & \text{if }i\notin H,\\
2qm+i & \text{if }i\in H.
\end{cases}
\]
Moreover, $S^{(q)}$ and $S^{(1)}$ lie in the relative interior of the same face of $\mathcal{C}_m$.
\end{proposition}

\begin{proof}
Since $m$ and $qm+1$ are generators, the generators have greatest common divisor $1$, and $S^{(q)}$ is a numerical semigroup of multiplicity $m$. Observe that the displayed generating set is minimal.

For $i\notin H$, the least element of residue $i$ is $qm+i$. For the missing residue/s, we have
\[
2qm+n=(qm+1)+(qm+n-1),
\]
and, when $H=\{2,n\}$, also $2qm+2=2(qm+1)$. An element of a missing residue involving at most one generator other than $m$ is impossible. If it involves two, their residues must sum to $i$ or $m+i$, and hence its value is at least $2qm+i$. If it involves at least three, its value is at least $3(qm+1)>2qm+i$. This proves the formula for the Ap\'ery set.

Write $h_i=1$ for $i\notin H$ and $h_i=2$ for $i\in H$. For $1\leq i,j\leq m-1$, let $k$ be the representative of $i+j$ modulo $m$ in $\{0,\ldots,m-1\}$. If $k\neq0$, then
\[
a_i^{(q)}+a_j^{(q)}-a_k^{(q)}
=qm(h_i+h_j-h_k)+(i+j-k).
\]
Both terms on the right are non-negative. Equality holds precisely when $i+j=k$, $i,j\notin H$, and $k\in H$. This condition is independent of $q$, so the Ap\'ery vectors satisfy exactly the same defining equalities of $\mathcal{C}_m$ for every $q\geq1$.
\end{proof}

\begin{corollary}\label{cor:shifted-sally-betti}
With the notation of Proposition \ref{prop:shifted-sally-face}, let $R^{(q)}=\KK[S^{(q)}]$, and let $Q$ be the presenting polynomial ring on the minimal generators of $S$. Then, for all $d\geq0$,
\[
\beta_d^Q(R^{(q)})
=\beta_d^Q(R^{(1)}),
\qquad
\beta_d^{R^{(q)}}(\KK)=\beta_d^{R^{(1)}}(\KK).
\]
In particular, the results of Theorem \ref{thm:mainA} hold for $S^{(q)}$. If $H=\{n\}$ and $m\geq4$, then
\[
\mathcal{P}_{\KK}^{R^{(q)}}(z)=
\begin{cases}
\displaystyle\frac{1+z}{1-(m-2)z} & \text{if }n=2,\\[6pt]
\displaystyle\frac{1+z}{1-(m-2)z+z^2} & \text{if }n\geq3.
\end{cases}
\]
Under these hypotheses, $R^{(q)}$ is Golod iff $n=2$. Moreover, if $n\geq3$, then $\gdi(R^{(q)})=n$ and $\mathcal{D}_n(R^{(q)})=1$.
\end{corollary}

\begin{proof}
Observe that $S^{(1)}$ equals $\Gamma_m(n)$ or $\Gamma_m(2,n)$. The equalities follow from Proposition \ref{prop:shifted-sally-face}, \cite[Corollary 4.3]{bgmns25}, and \cite[Proposition 3.5]{gosd24}. Since the embedding dimensions also agree, so do the Golod defect numbers. The remaining assertions follow from Theorems \ref{thm:residue-betti-gamma-n}, \ref{thm:Poincare-series-for-n=2}, and \ref{thm:golod}.
\end{proof}

\begin{example}\label{ex:shifted-one-gap}
For $S=\Gamma_8(5)$ and $q=2$, Proposition \ref{prop:shifted-sally-face} gives
\[
S^{(2)}=\langle8,17,18,19,20,22,23\rangle.
\]
Then $\KK[S^{(2)}]$ has Betti numbers $(1,20,64,90,65,23,3)$, and
\[
\mathcal{P}_{\KK}^{\KK[S^{(2)}]}(z)=\frac{1+z}{1-6z+z^2},
\]
as for $\KK[S]$.
\end{example}

\subsection{Gluing with \texorpdfstring{$\NN$}{N}}

Let $S=\langle g_1,\ldots,g_e\rangle$ be minimally generated. Choose an integer $a\geq2$ and a positive element $b\in S\setminus\{g_1,\ldots,g_e\}$ such that $\gcd(a,b)=1$. The numerical semigroup
\[
\widetilde S\colonequals aS+b\NN
=\langle ag_1,\ldots,ag_e,b\rangle
\]
is a \emph{gluing of $S$ and $\NN$}. The condition that $b$ is not a minimal generator ensures that the displayed generating set of $\widetilde S$ is minimal.

\begin{proposition}\label{prop:gluing-with-N}
With the notation above, let $R=\KK[S]$ and $\widetilde R=\KK[\widetilde S]$. Let $Q=\KK[x_1,\ldots,x_e]$ and $\widetilde Q=Q[y]$ (with appropriately scaled grading) be the respective presenting polynomial rings. Then
\[
\mathcal{P}_{\widetilde R}^{\widetilde Q}(z)
=(1+z)\mathcal{P}_R^Q(z),
\qquad
\mathcal{P}_{\KK}^{\widetilde R}(z)
=\frac{\mathcal{P}_{\KK}^R(z)}{1-z}.
\]
Consequently, for $d\geq0$,
\[
\beta_d^{\widetilde Q}(\widetilde R)
=\beta_d^Q(R)+\beta_{d-1}^Q(R),
\qquad
\beta_d^{\widetilde R}(\KK)=\sum_{j=0}^d\beta_j^R(\KK),
\]
where $\beta_{-1}(R)=0$.
\end{proposition}

\begin{proof}
The first formula follows from \cite[Corollary 3.2]{GimenezSrinivasan}, since $\KK[\NN]$ is a polynomial ring. Write $b=\sum_{i=1}^e u_i g_i$ and set $M=x_1^{u_1}\cdots x_e^{u_e}$. The gluing presentation gives
\[
\widetilde R\cong R[y]/(y^a-M).
\]
Since $a\geq2$ and $b\notin\{g_1,\ldots,g_e\}$, the monic polynomial $y^a-M$ is regular and lies in the square of the homogeneous maximal ideal. Thus \cite[Proposition 3.3.5(2)]{Avramov} gives
\[
\mathcal{P}_{\KK}^{\widetilde R}(z)
=\frac{\mathcal{P}_{\KK}^{R[y]}(z)}{1-z^2}
=\frac{(1+z)\mathcal{P}_{\KK}^R(z)}{1-z^2}
=\frac{\mathcal{P}_{\KK}^R(z)}{1-z}.
\]
The Betti number expressions follow by comparing coefficients.
\end{proof}

For either $S=\Gamma_m(n)$ or $S=\Gamma_m(2,n)$ in the parameter ranges considered here, both $m$ and $m+1$ belong to $\mathcal{A}(S)$. Therefore we may always take $a=2$ and $b=2m+1=m+(m+1)$, obtaining
\[
\widetilde S=2S+(2m+1)\NN.
\]
Proposition \ref{prop:gluing-with-N} and Theorem \ref{thm:mainA} produce the Betti numbers of the toric ideals of these gluings explicitly. For the one-gap family, we also obtain the following formulae for the residue field.

\begin{corollary}\label{cor:gluing-sally}
Let $m\geq4$, $2\leq n<m$, $S=\Gamma_m(n)$, and $\widetilde S=2S+(2m+1)\NN$. Then
\[
\mathcal{P}_{\KK}^{\KK[\widetilde S]}(z)=
\begin{cases}
\displaystyle\frac{1+z}{(1-z)(1-(m-2)z)} & \text{if }n=2,\\[6pt]
\displaystyle\frac{1+z}{(1-z)(1-(m-2)z+z^2)} & \text{if }n\geq3.
\end{cases}
\]
\end{corollary}

\begin{proof}
The formulae follow from Proposition \ref{prop:gluing-with-N} and Theorems \ref{thm:residue-betti-gamma-n} and \ref{thm:Poincare-series-for-n=2}.
\end{proof}

\begin{example}
For $S=\Gamma_8(5)$, the gluing is
\[
\widetilde S=2S+17\NN=\langle16,17,18,20,22,24,28,30\rangle.
\]
Then $\KK[\widetilde S]$ has Betti numbers $(1,21,84,154,155,88,26,3)$, and
\[
\mathcal{P}_{\KK}^{\KK[\widetilde S]}(z)
=\frac{1+z}{(1-z)(1-6z+z^2)}.
\]
\end{example}

\section{Further Questions and Conjectures}\label{sec:questions}

In \cite[Conjecture 6.4]{gsss25}, the authors conjecture that given $4\leq n_1\leq n_2<m-1$, the $i^{\text{th}}$ Betti numbers of $\Gamma_m(n_1,n_1+1)$ and $\Gamma_m(n_2,n_2+1)$ agree for all indices $i\leq n_1-3.$ We make a stronger conjecture, based on the techniques utilised in the proof of Lemma \ref{lem:alphad-2n}. Macaulay2 computations \cite{M2} verify the conjecture for small $m$ and $n$. 

\begin{conjecture}
    Let $4\leq n<m-1$, $S=\Gamma_m(n,n+1)$, and $R=\KK[S]$. Let $Q$ be a polynomial ring in $m-2$ variables and $I_S\subset Q$ be the defining ideal of $S$. Then the Betti numbers of $Q/I_S$ are $\beta_0=1$ and
    \[
    \beta_d=d\binom{m-2}{d+1}-2\binom{m-3}{d-1}+(n-1)\binom{m-n}{d-n+2}-(n-3)\binom{m-n-1}{d-n+1}\;\text{for } d\geq 1. \qedhere
    \]
\end{conjecture}

Although we compute the Betti numbers of the defining ideals of $\KK[\Gamma_m(2,n)]$, their Golodness remains to be analysed.

\begin{question}
    Let $4\leq n<m$, $S=\Gamma_m(2,n)$, and $R=\KK[S]$. 
    \begin{enumerate}[leftmargin=*, label=(\alph*)]
        \item When is $R$ Golod? If not, what is the smallest $i$ such that $\mathcal{D}_i(R)>0$?
        \item Is the Poincar\'e series $\mathcal{P}_\KK^R(z)$ a rational function?
        \item What are the Betti numbers of $\KK$ over $R$?
    \end{enumerate}
\end{question}

We end with some thoughts on the Golod defect numbers and Golod defect index introduced in Definition \ref{defn:Goloddefect}.

\begin{remark}\label{rmk:gdi}
    We believe that the Golod defect index might be an enlightening object of study for other classes of rings. Let $Q$, $I$, and $R$ be as in Definition \ref{defn:Goloddefect}. Recall that $R$ is Golod iff $\gdi(R)=\infty$. We note down some preliminary facts:
    \begin{enumerate}[label=(\alph*), leftmargin=*]
        \item By \cite[Proposition 4.2]{NguyenVeliche}, we have $\mathcal{D}_0(R)=\mathcal{D}_1(R)=\mathcal{D}_2(R)=0$. Hence, $\gdi(R)\geq3.$
        \item By \cite[Corollary 6.10]{Burke}, it follows that if $\text{codepth}(R)=c$, then $\gdi(R)\leq c+1$ or $\gdi(R)=\infty$.
        \item If $R$ is a complete intersection of embedding dimension $e$ and codimension $c\geq2$, then 
        \[
        \mathcal{P}_\KK^R(z)=\frac{(1+z)^e}{(1-z^2)^c} \quad \text{and} \quad \mathcal{P}^Q_R(z)=(1+z)^c.
        \]
        It follows that $\mathcal{D}_3(R)=\binom{c}{2}$ and therefore, $\gdi(R)=3$. Thus complete intersections take the least possible value of Golod defect index (excluding hypersufaces, which are always Golod).
        \item The converse of (c) does not hold, that is, there exist rings with Golod defect index three which are not complete intersections. For instance, the numerical semigroup ring $\KK[\Gamma_5(3)]$ is shown to have Golod defect index three in Theorem \ref{thm:golod}. However, the defining toric ideal of $\Gamma_5(3)$ has height three and a minimal generating set of size five (by Theorem \ref{thm:defining-ideal-for-n}). It follows that $\KK[\Gamma_5(3)]$ is not a complete intersection.
    \end{enumerate}
\end{remark} 

\begingroup
\renewcommand{\UrlFont}{\fontsize{8}{10}\selectfont\ttfamily}
\bibliographystyle{alpha}
\bibliography{references.bib}
\endgroup

\end{document}